\documentclass[10pt]{amsart} 

\usepackage[psamsfonts]{amssymb} 
\usepackage{amsfonts,amsmath} 
\usepackage{graphicx, color}
\usepackage{enumerate}
\usepackage{url}

\usepackage[all]{xy}

\newtheorem{thmA}{Theorem}

\newtheorem{theorem}{Theorem}[section]

\newtheorem{proposition}[theorem]{Proposition}
\newtheorem{lemma}[theorem]{Lemma}
\newtheorem{corollary}[theorem] {Corollary}

\newtheorem*{claim*}{Claim}

\theoremstyle{remark}
\newtheorem{remark}[theorem]{Remark}
\newtheorem{example}[theorem]{Example}

\theoremstyle{definition}
\newtheorem{definition}[theorem]{Definition}

\def\calp{\mathcal{P}}
\def\calg{\mathcal{G}}
\def\calh{\mathcal{H}}
\def\calk{\mathcal{K}}

\def\calx{\mathcal{X}}

\def\Z{\mathbb Z}

\def\G{\Gamma}

\def\aut{{\rm{Aut}}}
\def\out{{\rm{Out}}}
\def\outgh{{\out^0(A_\Gamma;\calg,\calh^t)}}
\def\autgh{{\aut^0(A_\Gamma;\calg,\calh^t)}}
\def\outghp{{\out^{[l]}(A_\Gamma;\calg,\calh^t)}}
\def\ia{{\rm{IA}_\G}}

\def\gln{{\rm{GL}}(n, \Z)}

\def\glnp{{\rm{GL}}(n, \Z/l\Z)}
\def\glm{{\rm{GL}}(m, \Z)}

\def\iap{\aut^{[l]}(A_\G)}
\def\auto{\aut^0(A_\G)}
\def\ion{{\overline{\rm{IA}}_\Gamma}}

\def\F{\mathbb{F}}
\def\gl{{\rm{GL}}}

\def\spine{{\mathcal{X}(\calg)}}
\def\vcd{{\rm{vcd}}}
\def\cd{{\rm{cd}}}
\newcommand\genby[1]{\langle#1\rangle}
\newcommand\ngenby[1]{\langle\langle#1\rangle\rangle}
\newcommand\abs[1]{\lvert#1\rvert}
\newcommand\setcomp[2]{\{\,#1\mid #2\,\}}
\def\<{\langle}
\def\>{\rangle}

\newcommand{\st}{\mathrm{st}}
\newcommand{\lk}{\mathrm{lk}}
\newcommand{\supp}{\mathrm{supp}}
\newcommand{\crsupp}{\mathrm{crsupp}}

\newcommand{\leqgd}{\leq_{\mathcal{G}_\Delta}}

\DeclareMathOperator{\im}{im}

\title{Relative automorphism groups of right-angled Artin groups}
\author{Matthew B. Day and Richard D. Wade}   
\begin{document}

\begin{abstract} 

We study the outer automorphism group of a right-angled Artin group $A_\G$ with finite defining graph $\G$. We construct a subnormal series for $\out(A_\G)$ such that each consecutive quotient is either finite, free-abelian, $\gln$, or a Fouxe-Rabinovitch group. The last two types act respectively on a symmetric space or a deformation space of trees, so that there is a geometric way of studying each piece. As a consequence we prove that the group $\out(A_\G)$ is type VF (it has a finite index subgroup with a finite classifying space).

The main technical work is a study of relative outer automorphism groups of RAAGs and their restriction homomorphisms, refining work of Charney, Crisp, and Vogtmann. We show that the images and kernels of restriction homomorphisms are always simpler examples of relative outer automorphism groups of RAAGs. We also give generators for relative automorphism groups of RAAGs, in the style of Laurence's theorem. 
\end{abstract}

\address{}
\email{}

\maketitle

\section{Introduction}

A right-angled Artin group (RAAG) $A_\G$ is a group $A_\G$ generated by vertices of a graph $\G$ with a commutator relation $[v,w]=1$ whenever $v$ and $w$ are connected by an edge in $\G$. Special cases include free groups (when the graph has no edges) and free abelian groups (when the graph is complete). Such groups are ubiquitous in geometric group theory, and there is a blossoming study of their automorphism groups. 

There is a popular mantra that as right-angled Artin groups interpolate between free and free-abelian groups, their outer automorphism groups should interpolate between $\out(\F_n)$ and $\gln$. Putting this idea into practice is harder: for example, in many cases $\out(A_\G)$ is a finite group \cite{CF, Day} and there are examples where $\out(A_\G)$ is infinite but virtually abelian \cite{BF}. However, there are common properties shared by $\out(A_\G)$ as $\G$ varies over all graphs. For instance, $\out(A_\G)$ is always virtually torsion free with finite virtual cohomological dimension \cite{CVFiniteness} and always satisfies the Tits alternative \cite{2014arXiv1408.0546H}. The purpose of this article is to show that these groups always have a common algebraic decomposition (Theorem~\ref{th:decomposition}) and relate this decomposition to the structure of their classifying spaces (Theorem~\ref{th:vf}).

Recall that a group is of \emph{type F} if it has a classifying space that is a CW complex with finitely many cells, and it is of \emph{type VF} if it has a finite index subgroup of type F. Geometrically, $\gln$ acts on a deformation space of marked tori (symmetric space) and $\out(\F_n)$ acts on a deformation space of marked metric graphs (Culler and Vogtmann's Outer space). These actions are neither free nor cocompact, however become free after passing to a finite index subgroup and can be made cocompact (by passing to the Borel--Serre bordification in the case of symmetric space, or the spine in the case of Outer space). As these spaces are contractible, they imply that $\gln$ and $\out(\F_n)$ are of type VF. 

More generally, when a group $G$ has a free product decomposition \[ G= G_1 \ast G_2 \ast \cdots G_k \ast \mathbb{F}_m ,\]  the group of outer automorphisms acting by conjugation on each $G_i$ is called the \emph{Fouxe-Rabinovitch group} of the decomposition and acts on a \emph{relative Outer space} introduced by Guirardel and Levitt \cite{GL2}.  Each relative outer space is contractible (it is a deformation space of trees in the sense of \cite{GLDeformation})  and retracts onto a cocompact spine. The action of the Fouxe-Rabinovitch group on relative Outer space is not proper in general, however simplex stabilizers are well-understood. As a consequence, if each $G_i$ and its center $Z(G_i)$ has a finite classifying space, Guirardel and Levitt show that the Fouxe-Rabinovitch group is of type VF. Key examples of Fouxe-Rabinovitch groups include $\out(\mathbb{F}_n)$ itself and the subgroup of basis conjugating automorphisms. Our main result shows that, up to finite index, $\out(A_\G)$ can be built out of copies of Fouxe-Rabinovitch groups, $\gln$, and free abelian groups.

\begin{thmA}\label{th:decomposition}
Let $\G$ be a finite graph. Then there is a finite index subgroup $\out^0(A_\G)$ of $\out(A_\G)$ with a finite subnormal series \[1=N_0\leq N_1 \leq \dotsm\leq N_k=\out^0(A_\G),\] such that each consecutive quotient $N_{i+1} / N_{i}$ is isomorphic to \begin{itemize}
                                                                                                                                                                                                                       \item a finitely generated free abelian group,
                                                                                                                                                                                                                       \item $\gln$ for some $n$, or
                                                                                                                                                                                                                       \item a Fouxe-Rabinovitch group given by a free factor decomposition of a special subgroup $A_\Delta$ of $A_\G$.
                                                                                                                                                                                                                      \end{itemize}

\end{thmA}

Following the discussion above, each of these pieces has a well-understood action on a Euclidean space, a symmetric space, or a deformation space of trees, respectively. 

 The \emph{principal congruence subgroup of level $l$} in $\out(A_\G)$ is the kernel of the induced action of $\out(A_\G)$ on $H_1(A_\G;\mathbb{Z}/l\mathbb{Z})$. This is a torsion-free subgroup when $l \geq 3$. 

\begin{thmA}\label{th:vf}
For $l \geq 3$, the principal congruence subgroup of level $l$ in $\out(A_\G)$ is of type F.
\end{thmA}

This is not quite a direct consequence of Theorem~\ref{th:decomposition}, but the classifying spaces for the principal congruence subgroups are built up out of the classifying spaces for (finite index subgroups of) the pieces. As each principal congruence subgroup is finite index in $\out(A_\G)$, this implies:

\begin{thmA}\label{th:f}
For every defining graph $\G$, the group $\out(A_\G)$ is of type VF.
\end{thmA}

Charney, Stambaugh, and Vogtmann have shown that the \emph{untwisted subgroup} $U(A_\G)$ of $\out(A_\G)$ acts properly and cocompactly on a contractible simplicial complex $K_\G$, which plays the same role as the spine of Outer space \cite{MR3626599}. This implies Theorem~\ref{th:vf} and Theorem~\ref{th:f} when $U(A_\G)$ is finite index in $\out(A_\G)$. This occurs when there is no pair of distinct vertices $v$,$w$ in $\G$ such that $\st(v) \subset \st(w)$. We do not make any assumptions on the defining graph.

\subsection{Relative automorphism groups}

The key to proving the above results is to work in the setting of \emph{relative automorphism groups}. These are natural generalizations of parabolic subgroups in $\gln$ and stabilizers of free factors in $\out(\F_n)$.

For any group $G$, if $\Phi$ is an element of $\out(G)$ and $H$ is a subgroup of $G$, we say that $\Phi$ \emph{preserves} $H$ (or $H$ is \emph{invariant} under $\Phi$) if there exists a representative automorphism $\phi \in \Phi$ such that $\phi(H)=H$, and we say that $\Phi$ \emph{acts trivially on} $H$ if there is a representative of $\Phi$ restricting to the identity on $H$. 
If $\calg$ and $\calh$ are families of subgroups of $G$, then the \emph{relative outer automorphism group} $\out(G;\calg,\calh^t)$ is the group of outer automorphisms which preserve each element of $\calg$ and act trivially on each element of $\calh$. As mentioned above, a classic example is the group of matrices in $\gln$ preserving a given flag in $\mathbb{Z}^n$. However, such automorphism groups have also appeared in the context of hyperbolic groups (where the peripheral structure comes from vertices of a JSJ decomposition \cite{MR2174093}) and in the study of automorphism groups of free groups (where $\calg$ can be a free factor system \cite{2013arXiv1302.2681H} or a set of conjugacy classes in the group \cite{GL}).  Importantly for us, Fouxe-Rabinvotich groups are relative automorphism groups, where each element $G_i$ in the free factor decomposition of $G$ is added to $\mathcal{H}$. A more detailed survey on  relative automorphism groups of RAAGs in the literature is given in Section~\ref{ss:previously}.

In this paper we study relative automorphism groups where $G=A_\G$ is a RAAG and the sets $\calg$ and $\calh$ consist of \emph{special subgroups} of $A_\G$. Each special subgroup is a RAAG $A_\Delta$ given by a full subgraph $\Delta$ of $\Gamma$. If $A_\Delta$ is a special subgroup that is invariant under $\out(G;\calg,\calh^t)$ then there is a restriction homomorphism 
\[R_\Delta \colon \out(A_\Gamma; \calg,\calh^t) \to \out(A_\Delta),\] 
which allows for inductive arguments based on the number of vertices in the graph. 
Charney, Crisp and Vogtmann introduced restriction homomorphisms in \cite{MR2372847}, where they studied automorphisms of triangle-free RAAGs. Going back to the results known for an arbitrary graph $\G$ mentioned at the beginning of the introduction, these methods were extended by Charney and Vogtmann to show that for any graph $\G$ the group $\out(A_\G)$ has finite virtual cohomological dimension \cite{CVFiniteness} and is residually finite \cite{CVSubQuot}. Restriction maps also form an important part of the proof of the Tits alternative for $\out(A_\G)$ (which was completed by Horbez \cite{2014arXiv1408.0546H} using the work in \cite{CVSubQuot}). Other recent results about automorphisms of RAAGs use invariance of special subgroups or restriction maps in an essential way (see, for example, \cite{HK,GS,Kielak,W1}).

In the remainder of the introduction, we will state our main results about relative automorphism groups, and sketch the key ideas which surround them and the proof of Theorem~\ref{th:decomposition}. In the sequel we mostly work in the finite index subgroup $\out^0(A_\G)$, which is generated by elements called \emph{inversions}, \emph{transvections}, and \emph{extended partial conjugations} (for the experts: an extended partial conjugation by $x$ is a product of partial conjugations on distinct components of $\G - \st(x)$).  We use $\outgh$ to denote the intersection of $\out^0(A_\G)$ with $\out(A_\G;\calg,\calh^t)$.  We first show that $\outgh$ is generated by the same types of elements as $\out^0(A_\G)$:

\begin{thmA}\label{th:generators}
If $\calg$ and $\calh$ are sets of special subgroups in $A_\G$, then the relative automorphism group $\outgh$ has a finite generating set consisting of the inversions, transvections, and extended partial conjugations which it contains.
\end{thmA}

The proof is given in Section~\ref{se:generators} and goes via Laurence's theorem \cite{Laurence}, which gives a generating set for the whole group $\out(A_\G)$. The key trick is to extend $A_\G$ to a larger RAAG $A_{\widehat \G}$ in such a way that $A_\G$ and every special subgroup of $\calg$ and $\calh$ is invariant under $\out^0(A_{\widehat\G})$.  
One property of this construction is that every automorphism in $\outgh$ extends to an element of $\out^0(A_{\widehat\G})$, and we can then deduce Theorem~\ref{th:generators} by restricting Laurence's generators of $\out^0(A_{\widehat\G})$ to $A_\Gamma$.

The following two observations form the backbone of the remainder of the paper:

\begin{itemize}
\item 
Given a RAAG $A_\G$, unless the group is free abelian or a free group, there is a nonempty set of proper special subgroups preserved by $\out^0(A_\G)$.
This means that, up to finite index, $\out(A_\G)$ already has an interesting peripheral structure and associated restriction homomorphisms. \item An arbitrary restriction homomorphism \[ R_\Delta \colon \out^0(A_\G;\calg,\calh^t) \to \out(A_\Delta) \] is not surjective in general, however we show that its image can be described as a relative outer automorphism group.
\end{itemize}

Let us describe the image of a restriction map $R_\Delta$ in more detail. Given sets of special subgroups $\calg$ and $\calh$, we say that $\calg$ is \emph{saturated} with respect to the pair $(\calg,\calh)$ if $\calg$ contains every proper special subgroup invariant under $\out^0(A_\G;\calg,\calh^t)$. Given a special subgroup $A_\Delta$, we use $\calg_\Delta$ to denote 
\[\calg_\Delta=\{A_{\Delta\cap\Theta} | A_\Theta\in\calg\}-\{A_\Delta\}.\] 
In words, $\calg_\Delta$ consists of the proper intersections of elements of $\calg$ with $A_\Delta$. We define $\calh_\Delta$ in the same fashion. 

\begin{thmA}\label{th:restriction}
Let $\mathcal{G}$ and $\mathcal{H}$ be sets of special subgroups in $A_\G$ such that $A_\Delta \in \calg$ and suppose that $\calg$ is saturated with respect to the pair $(\calg,\calh)$.
Then the restriction homomorphism \[ R_{\Delta} \colon \out^0(A_\G; \mathcal{G}, \mathcal{H}^t) \to \out(A_\Delta) \] has image equal to the relative automorphism group $$\im R_{\Delta} = \out^0(A_{\Delta};\calg_{\Delta}, \calh_{\Delta}^t)$$ and kernel equal to $$\ker {R_\Delta} = \out^0(A_\G; \calg, (\calh \cup \{A_{\Delta}\})^t ). $$
\end{thmA} 

Theorem~\ref{th:restriction} can be summarized as an exact sequence:
\[
\begin{split}
1\to \out^0(A_\G;\calg,(\calh\cup\{A_{\Delta}\})^t) \to  &\out^0(A_\G;\calg,\calh^t) \\
& \stackrel{R_{\Delta}}{\longrightarrow} \out^0(A_{\Delta};\calg_{\Delta},\calh_{\Delta}^t)\to 1.
\end{split}
\]
The description of the kernel of $R_\Delta$ is fairly straightforward, as is showing that the image of $R_\Delta$ is contained in $\out^0(A_{\Delta};\calg_{\Delta},\calh_{\Delta}^t)$. The hard work goes into showing that every element of $\out^0(A_{\Delta};\calg_{\Delta},\calh_{\Delta}^t)$ is contained in the image of $R_\Delta$. This is achieved by carefully lifting each generator of $\out^0(A_{\Delta};\calg_{\Delta},\calh_{\Delta}^t)$ to an element of $\out^0(A_\G;\calg,\calh^t)$.
We develop a clear picture of what the generators in each of these groups look like along the way (Proposition~\ref{pr:gensrelgh}).

To work examples, we also need to find images of restriction maps when $\calg$ is not saturated.
In Proposition~\ref{pr:easier_way_to_find_peripheral_structure}, we explain how to extend $\calg$ to a collection that is just big enough for the image in Theorem~\ref{th:restriction} to be correct.
This means that we do not need to compute saturations of sets of special subgroups by exhaustion, and it makes it possible to perform computations by hand on medium-sized examples.
We give examples in Section~\ref{s:examples} which the reader may find helpful to refer to for illustrations of key definitions and results. 

As a companion to Theorem~\ref{th:restriction}, in Section~\ref{s:decomposition} we determine what happens when a relative automorphism group has no nontrivial restriction map. In this case, it is either isomorphic to one of the three classes of groups listed in Theorem~\ref{th:decomposition} or there is a simplifying \emph{projection homomorphism}. This allows us to prove Theorem~\ref{pr:dismantleoutAGamma}, which states that every relative automorphism group has a subnormal series of the type given in Theorem~\ref{th:decomposition}, deducing the theorem as a special case. The proof follows an induction argument using an appropriate notion of complexity and the exact sequence given by Theorem~\ref{th:restriction}.

The proof of Theorem~\ref{th:vf} (each principal congruence subgroup has a finite classifying space) also proceeds in the relative setting and is again an induction argument. Here we use a `principal congruence subgroup' version of the exact sequence for restriction maps (see Theorem~\ref{th:pcongrestriction}), and couple this with the fact that type F is preserved under extensions by groups that are also of type F. Salvetti complexes for RAAGs, Guirardel and Levitt's Outer space for free products \cite{GL2} (with Culler--Vogtmann Outer space as a special case \cite{MR830040}), and the Borel--Serre bordification of symmetric space \cite{MR0387495} are used to handle the cases where all restriction maps are trivial.

We hope that our inductive scheme will be used to prove other results about automorphism groups of RAAGs.
For example, our approach seems to work well for computing virtual cohomological dimension of $\out(A_\G)$, in special cases.
We give two example computations in Section~\ref{ss:vcdeg}.


\thanks{The first author was supported by the National Science Foundation under
Grant No. DMS-1206981. Both authors were supported by the National Science
Foundation under Grant No. DMS-1440140 while in residence at the
Mathematical Sciences Research Institute in Berkeley, California, during
the Fall 2016 semester. Both authors would like to thank Benjamin Br\"uck and Andrew Sale for helpful comments on earlier versions of the paper.}

\section{Background}
Let $A_\G$ be a right-angled Artin group (RAAG). The presentation of this group is determined by the graph $\G$. There is a generator for each vertex $v$ in the vertex set $V(\G)$, and a relation $[v,w]=vwv^{-1}w^{-1}=1$ whenever vertices $v$ and $w$ are connected by an edge.  We will often blur the distinction between a vertex of the graph and the element of the group it represents.

For a subset $S$ of a group $G$, we write $C(S)$ and $N(S)$ for the centralizer and normalizer of $S$, respectively.
(When $S=\{g\}$ we write $C(g)$.) We use $\genby{S}$ and $\ngenby{S}$ to denote the subgroup generated by $S$ and the normal subgroup generated by $S$, respectively. For a subgroup $H$ of $G$, we use $Z(H)$ for the center of $H$.

\subsection{Special subgroups, links, and stars.}

When we talk about a subgraph $\Delta$ of $\G$ we will require that $\Delta$ is \emph{full}; any edge in $\G$ connecting two vertices in $\Delta$ is also an edge in $\Delta$. 

Given a vertex $v \in \G$, the \emph{link} $\lk(v)$ is the full subgraph of $\G$ spanned by vertices adjacent to $v$ in $\G$. The \emph{star} $\st(v)$ of $v$ is the full subgraph of $\G$ spanned by $\lk(v) \cup \{v\}$.  The \emph{link} of a subgraph $\Delta$ is the (possibly empty) intersection of the links of its vertices, like so: \[\lk(\Delta)=\bigcap_{v \in\Delta} \lk(v).\]
 The \emph{star} $\st(\Delta)$ of $\Delta$ is then the subgraph spanned by the vertices in $\lk(\Delta) \cup \Delta$.

Suppose $\Delta$ is a full subgraph of $\G$.
Sometimes we think of $A_\Delta$ as a RAAG in its own right, but more often, we identify $A_\Delta$ with the subgroup $\genby{\Delta}$ of $A_\Gamma$ generated by the vertices of $\Delta$.
This makes sense because the inclusion map $\Delta\to\Gamma$ defines an injective group homomorphism $A_\Delta\to A_\Gamma$ with image $\genby{\Delta}$.
The \emph{special subgroups} of $A_\Gamma$ are exactly the subgroups $A_\Delta$ as $\Delta$ varies over the subgraphs of $\Gamma$.

If $v$  is a vertex of $\Delta$, we can consider the link and star of $v$ both in $\Delta$ and in $\G$.
These can differ, and we use subscripts to distinguish these when necessary.
We write $\lk_\Delta(v)$, $\lk_\G(v)$, $\st_\Delta(v)$, and $\st_\G(v)$ for these.

For a subgraph $\Delta$ of $\Gamma$, let $Z(\Delta)$ denote the subgraph of $\Delta$ consisting of vertices that are adjacent to all other vertices of $\Delta$.
The following fact is a consequence of Servatius' Centralizer Theorem \cite{Servatius}, and has been noted often before (our phrasing follows Hensel--Kielak~\cite{HK}).

\begin{proposition}\label{pr:ZCN}
For a special subgroup $A_\Delta$ of $A_G$, we have
\begin{itemize}
\item $Z(A_\Delta)=A_{Z(\Delta)}$,
\item $C(A_\Delta)=A_{\lk(\Delta)\cup Z(\Delta)}$, and
\item $N(A_\Delta)=A_{\st(\Delta)}$.
\end{itemize}
\end{proposition}

\subsection{Words and elements}

We will look at words and word length in $A_\G$ with respect to the standard generating set $V(\G)$. A word $w$ representing an element of $g$ is \emph{reduced} if there exists no subword of the form $v^{\pm 1}w'v^{\mp 1}$, where $w'$ is a subword consisting of letters in $\st(v)$. The length of an element $g$ in the word metric is the length of any of its reduced representatives and one can pass between two reduced representatives by a sequences of swaps of consecutive commuting letters.  For $g\in A_\G$, the \emph{support} of $g$ is the smallest full subgraph $\supp(g)$ of $\G$ with $g\in\genby{\supp(g)}$.
An element $g$ is \emph{cyclically reduced} if it is of minimal word length among its conjugates. We will find it useful to have a notation for the support of the cyclic reduction of an element; we write $\crsupp(g)$ for this.

\subsection{The standard ordering on $V(\G)$}
\label{se:standardordering}

 The \emph{standard ordering} on the vertex set $V(\G)$ of $\G$ is the binary relation given by $u \leq v$ if $\lk(u) \subset \st(v)$. This relation is sometimes called \emph{domination}.
 This is a partial preorder, and induces an equivalence relation on the vertices where $u \sim v$ if $v \leq u$ and $u \leq v$. 
There is an induced partial ordering of equivalence classes in $V(\G)$, where we say that $[u] \leq [v]$ if for some (equivalently, any) representatives $v' \in [v]$ and $u' \in [u]$ of the equivalence classes we have $u' \leq v'$. A vertex $v$ in $\G$ is \emph{maximal} if the only vertices $w$ with $v\leq w$ are those with $v\sim w$. The vertices in each equivalence class $[v]$ generate either a free or a free abelian subgroup of $A_\G$. For proofs and more information about this partial order, see Section~2 of \cite{CVFiniteness}.

Suppose that $\Delta$ is full subgraph of $\G$. The standard ordering in $\Delta$ may not be the same thing as the induced ordering from $\G$.
We sometimes use subscripts to discriminate, writing $v\leq_\G w$ and $v\leq_\Delta w$. If $v\leq_\G w$ and $v$ and $w$ lie in $\Delta$ then $v\leq_\Delta w$, however the converse statement need not to be true.

\subsection{Generating sets and invariant subgroups}

 We use $\aut(A_\G)$ and $\out(A_\G)$ to denote the respective automorphism and outer automorphism groups of $A_\G$. We use $[\phi]$ to denote the image of an automorphism $\phi$ in $\out(A_\G)$, or a capital Greek letter (usually $\Phi$) to denote an arbitrary element of $\out(A_\G)$. A theorem of Laurence~\cite{Laurence} tells us that $\aut(A_\G)$ is generated by the following automorphisms:
\begin{itemize}
\item
\textbf{Graph symmetries.} 
Any automorphism of the graph induces an automorphism of the group via the corresponding permutation of the generating set. These elements of $\aut(A_\G)$ are called \emph{graph symmetries}.
\item
\textbf{Inversions.} If $v$ is a vertex of $\G$, then there is an \emph{inversion} $\iota_v$ that sends $v$ to $v^{-1}$ and fixes all other generators of $A_\G$.
\item
\textbf{Transvections.} Suppose $v$ and $w$ are vertices of $\G$ with $w\leq v$. 
There is a \emph{transvection} $\rho^{v}_w$ which takes $w$ to $wv$ and fixes all other generators of $A_\G$.
\item
\textbf{Partial conjugations.} Let $v$ be a vertex of $\G$ and let $C$ be a component of $\G - \st(v)$. There is a \emph{partial conjugation} $\pi^v_C$ which sends $w$ to $vwv^{-1}$ if $w$ is a vertex of $C$, and fixes each generator which is not a vertex of $C$.
\end{itemize}
For the examples of transvections and partial conjugations given above, we say that $v$ is the \emph{acting letter} of the automorphism. 
If $X$ is the union of connected components $C_1, C_2, \dotsc, C_k$  of $\G - \st(v)$, we refer to the product  $\pi^v_X=\pi^v_{C_1}\pi^v_{C_2} \cdots \pi^v_{C_k}$ of partial conjugations as an \emph{extended partial conjugation}, so that \emph{partial conjugation} is a special case where we conjugate by $v$ along a single connected component.  The \emph{extended Laurence generators} consist of the graph symmetries, inversions, transvections, and extended partial conjugations.

As is standard, we define $\aut^0(A_\G)$ to be the subgroup of $\aut(A_\G)$ generated by inversions, transvections and (extended) partial conjugations.
We use $\out^0(A_\G)$ to denote the image of $\aut^0(A_\G)$ in $\out(A_\G)$. 

We can read off the special subgroups invariant under a generator by the following lemma. Recall that $\phi$ \emph{acts trivially} on $A_\Delta$ if there exists $g \in A_\G$ such that $\phi(h)=ghg^{-1}$ for all $h \in A_\Delta$, and $\phi$ \emph{preserves} $A_\Delta$ if there exists $g \in A_\G$ such that $\phi(A_\Delta)=gA_\Delta g^{-1}$.

\begin{lemma}\label{le:genexercise}
Let $A_\Delta$ be a special subgroup of $A_\G$.
\begin{enumerate}
 \item The inversion $\iota_v$ acts trivially on $A_\Delta$ if and only if $v \not \in \Delta$; $\iota_v$ always preserves $A_\Delta$.
 \item The transvection $\rho^w_v$ acts trivially on $A_\Delta$ if and only if $v \not \in \Delta$; $\rho^w_v$ preserves $A_\Delta$ if and only if $v \not \in \Delta$ or both $v,w \in \Delta$.
 \item The extended partial conjugation $\pi^x_K$ acts trivially on $\Delta$ if and only if \begin{equation} \label{eq:star} \tag{$\ast$}  K \cap \Delta = \emptyset \text{ or } \Delta - \st(x) \subset K.\end{equation} The subgroup $A_\Delta$ is preserved if and only if $\Delta$ satisfies \eqref{eq:star} or $x \in \Delta$.
\end{enumerate}

\end{lemma}

\begin{proof}
We give a proof of (3) here, as parts (1) and (2) are similar. If $\Delta$ satisfies \eqref{eq:star} then either $\pi_K^x$ fixes every element of $A_\Delta$ or conjugates every element of $A_\Delta$ by $x$, so $\pi^x_K$ acts trivially on $A_\Delta$. If $\Delta$ does not satisfy \eqref{eq:star}, then there exist $u$ and $v$ in $\G - \st(x)$ such that $u \in K$ and $v \not \in K$. As $K$ is a union of connected components of $\G -\st(x)$, the vertices $u$ and $v$ are non-adjacent in $\G$, so the commutator $[u,v]$ is nontrivial in $A_\Delta$ and sent to \[[xux^{-1},v]=xux^{-1}vxu^{-1}x^{-1}v^{-1} \] under $\pi^x_K$. This element is cyclically reduced and therefore not conjugate to $[u,v]$, so the action of $\pi^x_K$ on $A_\Delta$ is nontrivial. Furthermore, if $x \not \in \Delta$ then $[xux^{-1},v]$ is not in any conjugate of $A_\Delta$ (as any cyclically reduced element of $gA_\Delta g^{-1}$ is contained in $A_\Delta$ itself). Hence if $x \not \in \Delta$ and $K$ does not satisfy \eqref{eq:star}, then $A_\Delta$ is not preserved by $\pi^x_K$. If $x \in \Delta$ then $A_\Delta$ is clearly preserved by $\pi^x_K$.
\end{proof}

If $\Delta$ is a subgraph of $\G$ then we say that $\Delta$ is \emph{upwards closed} if for every vertex $v \in \Delta$, every vertex $w$ with $v\leq w$ in the standard order is also contained in $\Delta$. We say that $\Delta$ is \emph{not star-separated by an outside vertex} if for every vertex $x \not \in \Delta$ there is at most one connected component $K$ of $\G - \st(x)$ such that $\Delta\cap K \neq \emptyset$. As was first observed in \cite{GS}, these two conditions describe the special subgroups invariant under $\out^0(A_\G)$ (one can also prove this via Lemma~\ref{le:genexercise}).

\begin{proposition}[\cite{GS}, Section 2.1] \label{pr:invspecial}
A special subgroup $A_\Delta$ of $A_\G$ is preserved by $\out^0(A_\G)$ if and only if $\Delta$ is upwards closed and not star-separated by an outside vertex.
\end{proposition}

\subsection{Restriction and projection homomorphisms}
 
Let $G$ be a subgroup of $\out(A_\G)$. If $A_\Delta$ is preserved by $G$ then there is a \emph{restriction homomorphism} \[ R_\Delta\colon G \to \out(A_\Delta), \] where $R_\Delta(\Phi)$ is defined by taking any representative $\phi$ of $\Phi$ that sends $A_\Delta$ to itself and restricting $\phi$ to $A_\Delta$. If the normal subgroup generated by $A_\Delta$ is fixed under $G$, then there is a projection map: \[ P_{\Gamma - \Delta} \colon G \to \out(A_{\G - \Delta})\] which is induced by the quotient map  \[ A_\G \to A_\G / \ngenby{A_\Delta} \cong A_{\G - \Delta}. \] An important feature of these maps is that each extended Laurence generator is taken either to the identity element or an extended Laurence generator of the same type under $R_\Delta$ and $P_{\Gamma - \Delta}$. It is not immediate that $R_\Delta$ is well-defined: this is a consequence of that fact that any element of the normalizer $N(A_\Delta)$ acts by conjugation as an inner automorphism of $A_\Delta$ (see Section 2.6 of \cite{W1} for more details).

\section{Relative automorphism groups.}\label{se:generators}

Let $\calg$ and $\calh$ be sets of proper special subgroups of the RAAG $A_\G$. 
In this section, we show Theorem~\ref{th:generators}:
 $\outgh$ is finitely generated by the inversions, transvections, and extended partial conjugations it contains.
We give a precise description of these generators in Proposition~\ref{pr:gensrelgh}, and in Proposition~\ref{pr:chariss} give criteria for when an arbitrary special subgroup $A_\Delta$ is preserved by $\out^0(A_\G;\calg,\calh^t)$.

\subsection{Passing from $\out(A_\G)$ to $\out^0(A_\G)$}

We first take a short detour to describe the cosets of $\out^0(A_\G)$ in $\out(A_\G)$ in detail, focusing on the action of $\out(A_\G)$ on conjugacy classes of special subgroups. Many of the results in this section are due to Duncan--Remeslennikov (see \cite{MR2999373}, particularly Theorem 4.4). The groups $A_{\geq v}$ (defined below) appear in \cite{MR2999373} as \emph{admissible subgroups}, and the relative automorphism group preserving all admissible subgroups appears as ${\rm{St}}^{\text{conj}}(\mathcal{K})$. We provide alternative proofs below as the ideas in this section are useful in the sequel. A description of how $\out^0(A_\G)$ sits inside $\out(A_\G)$ also appears in \cite{MR2372847}. 

For a given vertex $v$, we define associated special subgroups:
\begin{align*}
A_{\geq v} &= \genby{\setcomp{w}{v\leq w}} \\
A_{> v} &= \genby{\setcomp{w}{v \leq w \text{ and } v \not\sim w}}
\end{align*}
generated by the vertices in $\G$ that dominate $v$, and strictly dominate $v$, respectively. The group $A_{\geq v}$ is equal to $A_{\geq w}$ if and only if $v$ and $w$ are in the same equivalence class, and the quotient $A_{\geq v}/\ngenby{A_{>v}}$ is naturally isomorphic to the group $A_{[v]}$ generated by vertices equivalent to $v$. The conjugacy classes of the groups $A_{\geq v}$ are permuted by $\out(A_\G)$:

\begin{proposition}\label{p:gauto}
Let $\phi \in \aut(A_\G)$. There exists a graph automorphism $\sigma \in \aut(\G)$ such that $\phi(A_{\geq v})$ is conjugate to $A_{\geq \sigma(v)}$ for all $v$, and the normal subgroup $\ngenby{A_{> v}}$ generated by $A_{> v}$ is taken to $\ngenby{A_{> \sigma(v)}}$ under $\phi$.
\end{proposition}

\begin{proof}
If the statement holds for automorphisms $\phi$ and $\psi$ with respective graph automorphisms $\sigma$ and $\tau$, then there exist elements $g,h \in A_\G$ such that $\phi(A_{\geq v}) = g A_{\geq \sigma(v)} g^{-1}$ and $\psi(A_{\geq \sigma(v)})= h A_{\geq\tau\sigma(v)} h^{-1}$. Then
\begin{equation}\label{eq:hom}
\psi\phi(A_{\geq v}) = \psi(g)h \cdot A_{\geq \tau\sigma(v)} \cdot (\psi(g)h)^{-1},
\end{equation}
so that the statement holds for the product $\phi\psi$ with graph automorphism $\tau\sigma$. We use $A_{\geq v}$ in the above example, but the same reasoning holds for the normal subgroup generated by $A_{> v}$. Similarly, if the statement holds for $\phi$ with graph automorphism $\sigma$, it holds for $\phi^{-1}$ with graph automorphism $\sigma^{-1}$. Hence we only need to prove the result for generators of $\aut(A_\G)$.

If $\phi$ is a graph symmetry determined by $\sigma \in \aut(\G)$, then $\phi(A_{\geq v})= A_{\geq \sigma(v)}$ and $\phi(A_{> v})= A_{> \sigma(v)}$for each vertex $v$. If $\phi$ is an inversion or transvection then one can check that $\phi(A_{\geq v})=A_{\geq v}$ and $\phi(A_{> v})=A_{> v}$ for all vertices $v$. Finally, suppose that $\phi=\pi^x_K$ is an extended partial conjugation. If $x \in A_{\geq v}$ then $\pi^x_K$ preserves $A_{\geq v}$. Suppose $x \not \in A_{\geq v}$. Then $v \not \leq x$, so there exists $u \in \lk(v)$ that is not contained in $\st(x)$. As every vertex in $A_{\geq v}$ is adjacent to $u$, every vertex in $A_{\geq v}$ is contained in $\st(x)$ or the connected component $C$ of $\G - \st(x)$ containing $u$. As $K$ is a union of connected components of $\G - \st(x)$, either $\pi^x_K$ acts trivially on $A_{\geq v}$ or conjugates the whole group by $x$. Partial conjugations also preserve the normal subgroup generated by any special subgroup, so take the normal subgroup $\ngenby{ A_{> v} }$ to itself.  \end{proof}

Let $\overline{\G}=\G /\!\!\sim$ be the simple quotient of $\G$ obtained by identifying equivalent vertices and removing any duplicate edges or loops. 
In other words, $\overline{\G}$ has a vertex for each equivalence class $[v]$ and an edge between distinct equivalence classes $[v]$ and $[w]$ if $v$ and $w$ are connected by an edge in $\G$. There is a coloring $c$ of the vertices, where we assign a vertex the pair $c(v)=(\abs{[v]},0)$ if the group $A_{[v]}$ is abelian and $c(v)=(\abs{[v]},1)$ if the group $A_{[v]}$ is non-abelian. 
We call the colored graph $\overline{\G}$ the \emph{vertex class graph} of $\G$.
Let $\aut_c(\overline{\G})$ be the group of color-preserving automorphisms of the graph $\overline{\G}$. If $\sigma$ is an automorphism of $\G$, then $\sigma$ induces a color-preserving automorphism $\overline{\sigma}$ of the quotient graph $\overline{\Gamma}$.

\begin{proposition}\label{p:ghom}
There is a split-surjective homomorphism \[ \rho \colon \aut(A_\G) \to \aut_c(\overline{\G}), \] given by \[\rho(\phi)([v])=[w]\] where $w\in\G$ with $\phi(A_{\geq v})$ conjugate to $A_{\geq w}$.
\end{proposition}

\begin{proof}
Let $\phi \in \aut(A_\G)$, and let $\sigma$ be an automorphism of $\G$ given by Proposition~\ref{p:gauto}. Let $\overline{\sigma}$ be the induced action of $\sigma$ on the quotient $\overline{\G}$. We may then define $\rho(\phi)=\overline{\sigma}$. Although $\sigma$ is not uniquely determined by $\phi$, the quotient automorphism $\overline{\sigma}$ is well-defined, as the conjugacy class of each $A_{\geq v}$ is determined by the equivalence class $[v]$. The fact that $\rho$ is a homomorphism follows from equation (1) in the opening discussion in the proof of Proposition~\ref{p:gauto}. As $A_{\geq \sigma(v)}/\ngenby{A_{>\sigma(v)}} \cong A_{\geq v}/\ngenby{A_{>v}} \cong A_{[v]}$, the coloring of $\overline{\Gamma}$ is preserved by $\overline{\sigma}$. To see that $\rho$ is split-surjective, pick an ordering $v_1< v_2 < v_3 <\cdots <v_k$ on each equivalence class $[v]=\{v_1,\ldots,v_k\}$ and let $G$ be the finite group of order-preserving graph symmetries in $\aut(A_\G)$. One can check that the restriction of $\rho$ to $G$ is an isomorphism. 
\end{proof}

Each inner automorphism maps to the identity under $\rho$, therefore there is an induced map \[ \overline{\rho} \colon \out(A_\G) \to \aut_c(\overline{\G}). \]

The next proposition characterizes $\out^0(A_\G)$ as both the kernel of this map and as a relative automorphism group.

\begin{proposition}\label{pr:out0classification}
Let $\calg_{\geq}$ be the set of subgroups of the form $A_{\geq v}$ in $A_\G$, and let \[\rho\colon \aut(A_\G) \to \aut_c(\overline{\G}) \] be the homomorphism given by Proposition~\ref{p:ghom}. Then \[ \aut^0(A_\G) = \aut(A_\G;\calg_{\geq}) = \ker \rho , \] or equivalently \[ \out^0(A_\G) = \out(A_\G;\calg_{\geq}) = \ker \overline{\rho} . \]
\end{proposition}

\begin{proof}
It is enough to show that \[ \aut^0(A_\G) \subset \aut(A_\G;\calg_{\geq}) \subset \ker \rho \subset \aut^0(A_\G). \] For the first inclusion, suppose that $\phi \in \aut^0(A_\G)$. In the proof of Proposition~\ref{p:gauto} we showed that every generator of $\aut^0(A_\G)$ preserves each $A_{\geq v}$ up to conjugacy, therefore the same is true for $\phi$. Hence $\phi \in \aut(A_\G;\calg_{\geq})$.  For the second inclusion, suppose that $\phi \in  \aut(A_\G;\calg_{\geq})$. Then for each vertex $v$, the group $\phi(A_{\geq v})$ is conjugate to $A_{\geq v}$, so $\rho(\phi)([v])=[v]$. Hence $\rho(\phi)$ is the identity automorphism of $\overline{\Gamma}$. Finally suppose that $\phi \in \ker \rho$. Let $\iota_v$, $\rho^v_w$, and $\pi^v_K$ be respectively an inversion, transvection, and extended partial conjugation from the generating set of $\aut^0(A_\G)$. If we abuse notation slightly and use $\sigma$ to denote both an automorphism of the graph $\G$ and the induced automorphism of the group $A_\G$, we have \begin{align*} \sigma \cdot \iota_v&=\iota_{\sigma(v)} \cdot\sigma \\ \sigma \cdot\rho^v_w&=\rho^{\sigma(v)}_{\sigma(w)} \cdot\sigma \\ \sigma\cdot\pi^v_K&= \pi^{\sigma(v)}_{\sigma(K)}\cdot \sigma.\end{align*}
Hence by shuffling the generators of $\aut(A_\G)$, we can write $\phi$ as a product $\phi=\sigma\phi'$, where $\sigma$ is a graph automorphism and $\phi'$ is contained in $\aut^0(A_\G)$. Let $\overline{\sigma}$ be the automorphism of $\overline{\Gamma}$ induced by $\sigma$.  As $\phi' \in\aut^0(A_\G)$ we have $\rho(\phi')=1$ from the work above, therefore $\rho(\phi)=\overline\sigma$. As $\overline{\sigma}$ is the identity on $\overline{\G}$, this implies that $\sigma$ is a graphical automorphism that preserves each equivalence class (so $\sigma(v) \sim v$ for each vertex $v$). 
Hence $\sigma$, viewed as a permutation of the vertices, is a product of transpositions between elements in the same equivalence class. 
If $v$ and $w$ are equivalent, then one can verify that the product \[\iota_v(\rho_v^w)^{-1}\iota_v\iota_w\rho_w^v(\rho_v^w)^{-1} \] is the graphical automorphism swapping $v$ and $w$. 
It follows that each such transposition lies in $\aut^0(A_\G)$.
Therefore $\sigma \in \aut^0(A_\G)$, and so $\phi \in \aut^0(A_\G)$. This completes the chain of inclusions. \end{proof}

When $\phi$ is given to us in terms of images of generators, it is desirable to have a more explicit description of $\rho(\phi)$. We do this below.

\begin{proposition}\label{p:sigmaimage}
Let $\phi \in \aut(A_\G)$ and let $v \in V(\G)$. In the cyclic reduction support $\crsupp(\phi(v))$, there exists a vertex $w$ such that $w \leq u$ for all $u \in \crsupp(\phi(v))$. For any such vertex $w$ we have $\rho(\phi)([v])=[w]$.
\end{proposition}

\begin{proof}
Suppose that $\rho(\phi)([v])=[w]$. After modifying $\phi$ by an appropriate inner automorphism, we may assume that $\phi(A_{\geq v})=A_{\geq w}$. This does not change the cyclic reduction support of $\phi(v)$. By Proposition~\ref{p:gauto}, the normal subgroup generated by $A_{> v}$ is sent to $\ngenby{ A_{> w} }$. As $v$ is nontrivial in the quotient $A_{\geq v}/ \ngenby{A_{> v}}$, it follows that $\phi(v)$ is nontrivial in the quotient $A_{\geq w}/ \ngenby{ A_{> w} } \cong A_{[w]}$. Hence $\phi(v)$ must contain some vertex $w'$ equivalent to $w$ in its cyclic reduction support. As $\crsupp(\phi(v))$ is contained in $A_{\geq w}=A_{\geq w'}$, for every other vertex $u \in \crsupp(\phi(v))$ we have $w' \leq u$. As $w'$ is equivalent to $w$, we have $\rho(\phi)([v])=[w']$.
\end{proof}

\subsection{Generating $\out^0(A_\G;\mathcal{G})$}

Let $\calg$ be a collection of proper special subgroups of $A_\G$. 
It can be intuitively useful to think of groups in $\calg$ as representing the links in $A_\G$ of vertices that have been deleted from $\G$.
We build an extension of $\G$ to reflect this intuition.
We define a graph $\widehat \G$, called the \emph{relative cone graph} of $\G$ with respect to $\calg$.
This is the extension of $\G$ where we cone off each subgraph in $\calg$, and cone off $\G$ twice.
Specifically, let $\widehat \G$ be the extension of $\G$ by adding a new vertex $v_\Delta$ for each $A_\Delta\in\calg\cup\{A_\G\}$, and an additional vertex $v_*$.
We keep the edges of $\G$ in $\widehat\G$ exactly the same, add an edge between each vertex of $\Delta$ and $v_\Delta$, and an edge between every element of $\G$ and $v_*$.  Then $\lk_{\widehat\G}(v_\Delta)=\Delta$ for each $\Delta$, and $\lk_{\widehat\G}(v_*)=\G$.
Notice that $\G$ is a full subgraph of $\widehat \G$.
See the example in Figure~\ref{fig:lambda}.

\begin{figure}[ht]
\includegraphics{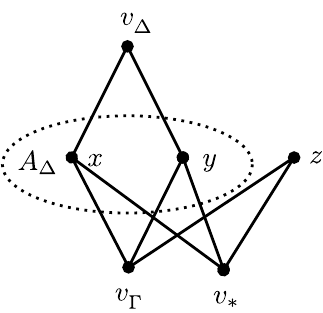} 
\caption{The relative cone graph $\widehat \G$ in the case that $A_\Gamma$ is the free group on $x,y,z$ and $\calg$ consists of a single peripheral subgroup $A_\Delta$ generated by $x$ and $y$.}
\label{fig:lambda}
\end{figure} 

We can now prove a special case of Theorem~\ref{th:generators}.

\begin{proposition}\label{pr:generating1}
The group $\out^0(A_\G;\calg)$ is generated by inversions, transvections and extended partial conjugations.
\end{proposition}

\begin{proof}
We use the relative cone graph $\widehat \G$ that we just defined.
We view $A_\G$ and each peripheral subgroup $A_\Delta$ as special subgroups of $A_{\widehat \G}$.

\begin{claim*}
If $A_\Delta \in \calg\cup\{A_\G\}$ then $A_\Delta$ is invariant under $\out^0(A_{\widehat \G})$.
\end{claim*}

Let $A_\Delta$ be a special subgroup in $\calg\cup\{A_\G\}$. 
According to Proposition~\ref{pr:invspecial}, we need to check that $\Delta$ is upwards-closed under the partial ordering $\leq_{\widehat\G}$ and not star-separated by any outside vertex. Suppose $v \in \Delta$ and $v \leq_{\widehat\G} w$ for some vertex $w \in \widehat\G$. As both $v_*$ and $v_\Delta$ are in $\lk_{\widehat\G}(v)$, it follows that $w$ lies in the intersection of $\st_{\widehat \G}(v_\Delta)$ and $\st_{\widehat \G}(v_*)$, which is $\Delta$. Hence $\Delta$ is upwards-closed under the partial order.

Now suppose $x \not \in \Delta$. 
If $x \neq v_\Delta$ then $\Delta - \st_{\widehat\G}(x)$ is contained in the connected component of ${\widehat\G} - \st_{\widehat\G}(x)$ containing $v_\Delta$. 
Otherwise, $x=v_\Delta$ and $\Delta$ is contained in $\st_{\widehat\G}(x)$. 
In either case, $\Delta$ intersects at most one component of ${\widehat\G} - \st_{\widehat\G}(x)$, so is not star-separated by $x$. 
Hence $A_\Delta$ is invariant under $\out^0(A_{\widehat\G})$. 

\begin{claim*}
The restriction map \[ R_\G\colon \out^0(A_{\widehat\G}) \to \out(A_\G) \] has image equal to $\out^0(A_\G; \calg)$.
\end{claim*}

Each inversion, transvection, and extended partial conjugation in $\out(A_{\widehat\G})$ is mapped either to the identity element of $\out(A_\G)$ or to an automorphism of the same type in $\out(A_\G)$. Hence the image is contained in $\out^0(A_\G)$. Furthermore, as each $A_\Delta \in \calg$ is preserved by $\out^0(A_{\widehat\G})$, the restriction of an automorphism to $A_\Gamma$ also preserves such subgroups. Hence the image is contained in $\out^0(A_\G;\calg)$. 
In order to show surjectivity, we will take $\phi \in \aut^0(A_\G;\calg)$ and extend $\phi$ to an automorphism $\tilde \phi$ of $A_{\widehat\G}$. 
For each $A_\Delta \in \calg$ we first pick $g_\Delta \in A_\Gamma$ such that $\phi(A_\Delta)=g_\Delta A_\Delta g_\Delta^{-1}$. We define $\tilde \phi$ on vertices by:
\[
\tilde \phi (v) =
\left\{
\begin{array}{cl}
\phi(v) &\text{if $v\in \Gamma$}\\
g_\Delta v g_\Delta^{-1}&\text{if $v=v_\Delta$ and $\Delta \neq \Gamma$}\\
v &\text{if $v$ is $v_\G$ or $v_*$.}
\end{array}
\right.
\]

To check that $\tilde \phi$ determines a well-defined homomorphism, note that if $v$ and $w$ are connected by an edge in ${\widehat\G}$ then either $v$ and $w$ are both in $\Gamma$ (in which case $\phi(w)$ and $\phi(v)$ commute), or $v\in\{v_\G,v_*\}$ and $w\in\G$ (in which case $v$ commutes with $\phi(w)$), or $v$ is some $v_\Delta$ and $w \in \Delta$. In the last case $\tilde \phi(v)=g_\Delta v_\Delta g_\Delta^{-1}$ and $\tilde\phi(w)= \phi(w) = g_\Delta h g_\Delta^{-1}$ for some $h \in A_\Delta$. As $v_\Delta$ commutes with every element of $A_\Delta$, it follows that $\tilde \phi(v)$ and $\tilde \phi(w)$ also commute.
We may also lift $\phi^{-1}$ to a homomorphism $\tilde\psi$ on $A_{\widehat \G}$ by defining 
\[
\tilde \psi (v) =
\left\{
\begin{array}{cl}
\phi^{-1}(v) &\text{if $v\in \Gamma$}\\
\phi^{-1}(g_\Delta^{-1}) v \phi^{-1}(g_\Delta)&\text{if $v=v_\Delta$ and $\Delta \neq \Gamma$}\\
v &\text{if $v$ is $v_\G$ or $v_*$.}
\end{array}
\right.
\]

One can verify that $\tilde \phi \tilde \psi (v) = \tilde \psi\tilde{\phi}(v) =v$ for every vertex of ${\widehat \G}$, so these lifts are inverses of each other, and $\tilde \phi$ is an automorphism.

Unfortunately we still need to show that $\tilde \phi \in \aut^0(A_{\widehat \G})$ (or in the language of the previous section, $\rho(\tilde\phi)=1$). 
Let $v \in {\widehat \G}$. By Proposition~\ref{p:sigmaimage}, there exists $w$ in $\crsupp(\tilde\phi(v))$ such that $w \leq_{\widehat \G} u$ for all $u \in \crsupp(\tilde\phi(v))$. 
Also by Proposition~\ref{p:sigmaimage}, to show that $\rho(\tilde\phi)=1$ it is enough to show that this vertex $w$ is equivalent to $v$ with respect to $\leq_{\widehat \G}$. 
If $v$ is some $v_\Delta$ or $v_\G$ or $v_*$, then $v$ is the only vertex in $\crsupp(\tilde\phi(v))$, so the case we need to consider is when $v \in \Gamma$. 
Then $\tilde\phi(v)=\phi(v)$, so any such $w$ in $\crsupp(\tilde\phi(v))$ is contained in $\G$. 
Furthermore, $w \leq_{\widehat \G} u$ for all $u \in \crsupp(\tilde\phi(v))$ implies that $w \leq_\Gamma u$ for all $u \in \crsupp(\phi(v))$. 
As $\phi \in \aut^0(A_\G)$, Proposition~\ref{p:sigmaimage} implies that $v \sim_{\G} w$. 
We need to improve this to show that $v \sim_{\widehat \G} w$. 
This is equivalent to showing that for each $A_\Delta \in \calg$, we have $v \in A_\Delta$ if and only if $w \in A_\Delta$ (this implies that $v_\Delta \in \st_{\widehat \G}(v)$ if and only if $v_\Delta \in \st_{\widehat \G}(w)$). For the first direction, suppose that $v \in A_\Delta$. 
As $A_\Delta$ is preserved by $\phi$, we have $\crsupp(\phi(v))\subset A_\Delta$. 
Hence $w \in A_\Delta$. Now suppose that $w \in A_\Delta$. We argue by contradiction and assume that $v \not\in A_\Delta$. As $A_\Delta$ is preserved by $\phi$, there is an induced automorphism 
\[ \bar\phi\colon A_\Gamma / \ngenby{A_\Delta} \to A_\Gamma / \ngenby{A_\Delta} .\] 
As $v$ is nontrivial in $A_\Gamma / \ngenby{A_\Delta}$, so is $\phi(v)$, so there exists a vertex $u$ in $\crsupp(\phi(v))$ that is not contained in $A_\Delta$. 
Then $\lk_{\widehat \G}(w)$ is not contained in $\st_{\widehat \G}(u)$ as $\lk_{\widehat \G}(w)$ contains $v_\Delta$ and $\st_{\widehat \G}(u)$ does not. 
This contradicts $w$ being minimal with respect to $\leq_{\widehat \G}$ in $\crsupp(\phi(v))$. 
Hence $w \sim_{\widehat \G} v$ and $\tilde \phi \in \aut^0(A_{\widehat \G})$.

Hence for each $\phi \in \aut^0(A_\Gamma,\calg)$ we have built an element $\tilde\phi \in \aut^0(A_{\widehat \G})$ which restricts to  $\phi$ on $A_\Delta$. As $\out^0(A_{\widehat \G})$ is generated by inversions, transvections, and extended partial conjugations and these map to either the identity or elements of the same type in $\out^0(A_\Gamma;\calg)$, it follows that $\out^0(A_\Gamma;\calg)$ is also generated by such elements.
 \end{proof}

\subsection{Relative connectivity and generators in $\out^0(A_\Gamma;\calg,\calh^t)$} 

Before proving Theorem A in full generality, we first need to describe exactly which extended Laurence generators lie in $\outgh$. 

\subsubsection{Relative connectivity}

Let $\calg$ be a collection of proper special subgroups of $A_\Gamma$ (so that $A_\G \not \in \calg$). Two vertices $v,w\in V(\G)$ are \emph{$\calg$-adjacent} if $v$ is adjacent to $w$ or if there is some $A_\Delta\in\calg$ with $v,w\in \Delta$. A finite sequence of vertices is a \emph{$\calg$-path} if each vertex in the sequence is $\calg$-adjacent to the next. A full subgraph $\Delta$ of $\G$ is \emph{$\calg$-connected} if, for any two $v,w\in V(\Delta)$, there is a $\calg$-path in $\Delta$ from $v$ to $w$. Each maximal $\calg$-connected subgraph is a union of connected components of $\G$, which in the same vein as above, we call a \emph{$\calg$-component}.
Notice that $\varnothing$-adjacency, $\varnothing$-paths, $\varnothing$-connectedness, and $\varnothing$-components are the same as adjacency, paths, connectedness, and connected components, respectively.

\subsubsection{Ordering and star-separation in the relative setting}

\begin{definition}[The $\calg$--ordering, $\calg^v$-components, and $P(\mathcal{H})$]
Let $\calg$ be a set of special subgroups of $A_\G$.
\begin{itemize}
\item We define the partial preorder $\leq_\calg$ on $V(\G)$ by saying that $v\leq_\calg w$ if and only if $v\leq_\G w$ and, for all $A_\Delta\in \calg$, if $v\in \Delta$, then $w\in\Delta$. We call this the \emph{$\calg$--ordering}. 
As usual, we have an equivalence relation $\sim_\calg$ defined by
$v\sim_\calg w$ if and only if $v\leq_\calg w$ and $w\leq_\calg v$.

\item For $v\in V(\G)$, let $\calg^{v}=\{A_\Delta\in \calg | v\not\in \Delta\}$ be the subset of $\mathcal{G}$ consisting of the peripheral subgroups that \emph{do not} contain $v$. Each $\calg^{v}$-component of $\G-\st(v)$ is a union of connected components of $\G - \st(v)$.

\item If $\mathcal{H}$ is a set of special subgroups, we define $P(\mathcal{H})$, the \emph{power set} of $\mathcal{H}$, to be the set of special subgroups $A_\Delta$ such that $A_\Delta$ is contained in some element of $\mathcal{H}$.
\end{itemize}
\end{definition}

Each element $A_\Delta \in \calg^v$ is contained in $A_{\lk(v)}$ or intersects exactly one $\calg^v$-component of $\G -\st(v)$. In fact, this is a defining feature of such components:

\begin{lemma}\label{le:calgcomponents}
If $K$ is a subgraph of $\G - st(v)$ then $K$ is a union of $\calg^{v}$-components if and only if: \begin{itemize}
                                                                                                    \item $K$ is a union of connected components of $\G - \st(v)$ and,
                                                                                                    \item for each $A_\Delta \in \calg^{v}$, either $\Delta \cap K = \emptyset$ or $\Delta - \st(v) \subset K$.
                                                                                                   \end{itemize}

\end{lemma}

\begin{proof}
The two conditions are equivalent to saying that if $w \in K$, then every vertex $\calg^{v}$-adjacent to $w$ is also in $K$. This is exactly the requirement for $K$ to be a union of $\calg^{v}$-components of $K$. 
\end{proof}

When working in $\outgh$ the set $\mathcal{G}$ can always be extended to include $P(\calh)$. This is because if $A_\Delta \in \calh$ and $\Lambda$ is a subgraph of $\Delta$, then any automorphism that restricts to the identity on $A_\Delta$ also restricts to the identity automorphism on $A_\Lambda$. We record this observation as a lemma below:

\begin{lemma}\label{le:calhsubsets}
Suppose $\calg$ and $\calh$ are sets of special subgroups and $A_\Lambda \in P(\calh)$. Then every element of $\out^0(A_\G;\calg,\calh^t)$ preserves $A_\Lambda$. In particular, \[ \out^0(A_\G;\calg\cup P(\calh),\calh^t)=\out^0(A_\G;\calg,\calh^t). \]
\end{lemma}

Recall that in $\aut(A_\G)$, inversion automorphisms $\iota_v$ are well-defined for any vertex $v$, and the transvection $\rho_v^w$ determines a well-defined automorphism if and only if $v \leq w$ with respect to the standard order. Furthermore the extended partial conjugation $\pi^v_K$ exists if and only if $K$ is a union of connected components of $\G - st(v)$. There are similar conditions in the relative setting:

\begin{proposition}[Generators in relative automorphism groups]\label{pr:gensrelgh}
Suppose that $\calg$ and $\calh$ are sets of proper special subgroups of $\aut(A_\G)$ and $\calg$ contains $P(\calh)$.
Then:
\begin{itemize}
\item Let $v\in V(\G)$. 
The inversion $\iota_v$ is in $\aut^0(A_\G;\calg,\calh^t)$ if and only if there is no subgroup $A_\Delta\in \calh$ with $v\in\Delta$.
\item Let $v,w\in V(\G)$.
The transvection $\rho^w_v$ is in $\aut^0(A_\G;\calg,\calh^t)$ if and only if $v\leq_\calg w$.
\item Suppose $v\in V(\G)$ and $K$ is a union of components of $\G-\st(v)$.
Then the extended partial conjugation $\pi^v_K$ is in $\aut^0(A_\G;\calg,\calh^t)$ if and only if $K$ is a union of $\calg^v$-components of $\G - \st(v)$.
\end{itemize}
\end{proposition}

\begin{proof}
The statement about inversions is equivalent to part (1) of Lemma~\ref{le:genexercise}. We fix $\rho^w_v$ to be a well-defined transvection (so that $v$ and $w$ are distinct and $v \leq_\Gamma w$). We use part (2) of Lemma~\ref{le:genexercise}. Firstly suppose that $v \leq_\calg w$. If $A_\Delta \in \calg$ then $\rho^w_v$ preserves $A_\Delta$ as either $v \not \in \Delta$ or both $v$ and $w$ are contained in $A_\Delta$ from the definition of $\leq_\calg$. We claim that $v$ is not contained in any element of $\calh$. Indeed, if this was the case, then $A_{\{v\}}$ is contained in $\calg$ as $\calg$ contains $P(\calh)$. This is an element of $\calg$ containing $v$ and not $w$, which contradicts the assumption that $v \leq_\calg w$. As $v$ is not contained in any element of $\calh$ the transvection acts trivially on these subgroups, so $\rho^w_v$ is contained in $\autgh$. Conversely, if $v \not \leq_\calg w$ then there exists $A_\Delta \in \calg$ containing $v$ and not $w$. In this case $A_\Delta$ is not preserved by $\rho^w_v$ and the transvection is not an element of $\autgh$. 

Now assume that $\pi^v_K$ is a extended partial conjugation, so that $K$ is a union of connected components of $\G - \st(v)$. By Lemma~\ref{le:genexercise}, the automorphism  $\pi^v_K$ acts trivially on $A_\Delta$ if and only if:
\begin{equation*} \tag{$\ast$}  K \cap \Delta = \emptyset \text{ or } \Delta - \st(v) \subset K.\end{equation*} Furthermore, $\pi^v_K$ preserves $A_\Delta$ if and only if $v \in \Delta$ or \eqref{eq:star} holds.

Suppose that $K$ is a union of $\calg^{v}$-components of $\G - \st(v)$. If $A_\Delta \in \calg$, then either $v \in A_\Delta$ or $A_\Delta \in \calg^v$ and $\Delta$ satisfies \eqref{eq:star} by Lemma~\ref{le:calgcomponents}. In either case, $A_\Delta$ is preserved by $\pi^v_K$. If $A_\Delta \in \calh$, then $\Theta=\Delta- \{v\}$ lies in $\calg$ as $\calg$ contains all elements of $P(\calh)$. As $v \not \in \Theta$, the group $A_\Theta \in \calg^v$  and satisfies \eqref{eq:star}. Hence $\pi^v_K$ restricts to an inner automorphism on $A_\Theta$, so also restricts to an inner automorphism on $A_\Delta = \genby{A_\Theta,v}$. Putting the above together gives $\pi^v_K \in \autgh$. Conversely, if $K$ is not a union of $\calg^{v}$-components, there exists $A_\Delta \in \calg^{v}$ which does not satisfy \eqref{eq:star}, which implies $A_\Delta$ is not preserved by $\pi^v_K$ and $\pi^v_K$ is not an element of $\autgh$.
\end{proof}

\begin{remark}\label{re:notPh}
In the proof of Proposition~\ref{pr:gensrelgh}, we use the hypothesis that $\calg$ contains $P(\calh)$.
We can weaken that hypothesis to this: $\calg$ should contain $\calh$, and should include every special subgroup of the form $A_{\Delta-\{v\}}$, for $A_\Delta\in\calh$ and $v\in\Delta$.
This gives us the correct transvections, since if $v,w\in\G$ with $v\in \Delta$ for some $A_\Delta\in\calh$, then $A_{\Delta-\{w\}}\in\calg$, and therefore $v\not\leq_\calg w$.
It also gives us the correct extended partial conjugations, since the only groups from $P(\calh)$ that we use in that part of the proof are those of the form $A_{\Delta-\{v\}}$.
\end{remark}

\subsection{A generating set for $\outgh$}

We can now complete our proof of Theorem~\ref{th:generators}, which states that $\outgh$ is generated by the inversions, transvections, and extended partial conjugations it contains.

\begin{proof}[Proof of Theorem~\ref{th:generators}]
By Lemma~\ref{le:calhsubsets}, we may assume that $\calg$ contains $P(\calh)$. 
By Proposition~\ref{pr:gensrelgh}, when $\calg$ contains $P(\calh)$ the only extended Laurence generators in $\out^0(A_\G;\calg)$ that are not contained in $\out^0(A_\Gamma;\calg,\calh^t)$ are inversions $\iota_x$ with $x$ contained in some element of $\calh$ (the conditions for transvections and partial conjugations use $\calg$ alone). 
Let $\phi \in \out(A_\G;\calg,\calh^t)$. As $\phi \in \out(A_\G;\calg)$, by Proposition~\ref{pr:generating1}, we may write $\phi$ as a product of inversions, transvections, and extended partial conjugations in $\out^0(A_\Gamma;\calg)$. 
By reshuffling this product using the identities 
\begin{align*} 
\iota_x\iota_y &=\iota_y\iota_x &\\
\iota_x\rho^w_v &= \begin{cases}(\rho^w_v)^{-1}\iota_x &\text{if $x=w$} \\
\rho^w_v\iota_x &\text{if $x \neq w$},\end{cases} \\
\iota_x\pi^v_K &=\begin{cases} (\pi^v_K)^{-1}\iota_x &\text{if $x=v$} \\
\pi^v_K\iota_x &\text{if $x \neq v$},\end{cases}
\end{align*}
 and using the fact that inversions commute and have order 2, we have \[ \phi=\iota_{x_1}^{\epsilon_1}\cdot \iota_{x_2}^{\epsilon_2} \cdots \iota_{x_k}^{\epsilon_k} \cdot \phi', \] where $x_1,\ldots,x_k$ are the vertices of $\Gamma$ that are contained in some element of $\calh$, each $\epsilon_i$ is either equal to $0$ or $1$, and $\phi'$ is a product of extended Laurence generators that are contained in $\outgh$. It remains to show that each $\epsilon_i=0$. As $\phi'$ acts trivially on the elements of $\calh$, for each $i$ we have $\phi'(x_i)=g_i x_i g_i^{-1}$ for some $g_i \in A_\Gamma$. If $\psi=\iota_{x_1}^{\epsilon_1}\cdot \iota_{x_2}^{\epsilon_2} \cdots \iota_{x_k}^{\epsilon_k}$, then \[ \phi(x_i)=\psi(g_i) x_i^{(-1)^{\epsilon_i}} \psi(g_i^{-1}).\] As $\phi$ restricts to an inner automorphism on each $\genby{x_i}$, it follows that $\epsilon_i=0$, and $\phi$ is a product of inversions, transvections, and extended partial conjugations in $\outgh$. 
\end{proof}

The above proof shows that the group $\out^0(A_\Gamma;\calg,\calh^t)$ is finite index in the group $\out^0(A_\G;\calg \cup P(\calh))$, so for many purposes one can happily consider relative automorphism groups without the extra set of special subgroups admitting a trivial action. In the sequel it is convenient for us to remain in the more general setting.

\subsection{Invariant special subgroups under $\outgh$}

We may now classify when an arbitrary special subgroup $A_\Delta$ is invariant under $\out^0(A_\G;\calg,\calh^t)$.

\begin{proposition}[Invariant special subgroups]\label{pr:chariss}
Suppose that $\calg$ and $\calh$ are sets of proper special subgroups of $A_\G$ and $\calg$ contains $P(\calh)$. Let $A_\Delta$ be any special subgroup of $A_\G$.
Then $A_\Delta$ is invariant under $\out^0(A_\G; \calg,\calh^t)$ if and only if the following two conditions hold:
\begin{itemize}
\item for all $v,w\in V(\G)$, if $v\in \Delta$ and $v\leq_\calg w$, then $w\in\Delta$, and
\item for all $v\in V(\G)$, if $v \not \in\Delta$, then $\Delta$ intersects at most one $\calg^v$-component of $\G - \st(v)$.
\end{itemize}
\end{proposition}

\begin{proof}
If $\Delta$ fails the first condition, pick vertices $v,w$ such that $v \leq_\calg w$ with $v \in \Delta$ and $w \not \in \Delta$. Then the transvection $\rho^w_v$ is contained in $\outgh$ (Proposition~\ref{pr:gensrelgh}) but $A_\Delta$ is not invariant under $\rho^w_v$ (Lemma~\ref{le:genexercise}). Similarly if $\Delta$ intersects more than one $\calg^{v}$-component of $\G - \st(v)$, if we label one of those components as $K$, the automorphism $\pi^v_K$ lies in $\autgh$ and does not preserve $A_\Delta$. Conversely, if $\Delta$ satisfies the above conditions then every extended Laurence generator in $\autgh$ leaves $A_\Delta$ invariant (again combining Proposition~\ref{pr:gensrelgh} and Lemma~\ref{le:genexercise}), so Theorem~\ref{th:generators} implies that $A_\Delta$ is preserved by the whole group.
\end{proof}

We say that \emph{$v$ is a vertex that $\calg$--star-separates $\Delta$} if the subgraph $\Delta$ intersects more than one $\calg^v$-component of $\G-\st(v)$. To put the conditions of Proposition~\ref{pr:chariss} into words, a subgroup $A_\Delta$ is invariant under $\outgh$ if and only if $\Delta$ is \emph{upwards closed under the order $\leq_\calg$}  and is \emph{not $\calg$--star-separated by any outside vertex}.

\subsection{Consequences of Proposition~\ref{pr:chariss}}

The next fact is essentially the same as Hensel--Kielak~\cite{HK}, Lemma~4.2, part (i).

\begin{corollary}\label{co:preserveintersections}
Suppose $\out^0(A_\G;\calg)$ preserves both
 $A_{\Delta_1}$ and $A_{\Delta_2}$.
Then it also preserves $A_{\Delta_1\cap\Delta_2}$.
\end{corollary}

\begin{proof}
Let $v, w\in V(\G)$ with $v\in \Delta_1\cap\Delta_2$ and $v\leq_\calg w$.
Then $v\in \Delta_i$, and since $A_{\Delta_i}$ is preserved, we have $w\in \Delta_i$, by Proposition~\ref{pr:chariss}, for $i=1,2$.
Then $w\in \Delta_1\cap\Delta_2$ and the intersection is upwards closed under $\leq_\calg$.

Now suppose that $w$ is a vertex that $\calg$--star-separates $\Delta_1 \cap \Delta_2$, so that there exist $u,v\in \Delta_1\cap\Delta_2$ in different $\calg^w$-components of $\G -\st(w)$.
As each $\Delta_i$ is not $\calg$--star-separated by an outside vertex (Proposition~\ref{pr:chariss}), it follows that $w \in \Delta_1$ and $w \in \Delta_2$.
Then $w\in \Delta_1\cap\Delta_2$. Hence if $w \not \in \Delta_1 \cap \Delta_2$ then $\Delta_1 \cap \Delta_2$ intersects at most one $\calg^w$-component of $\G - \st(w)$.

By Proposition~\ref{pr:chariss} the group  $\out^0(A_\G;\calg)$ preserves $A_{\Delta_1\cap\Delta_2}$.
\end{proof}

Proposition~\ref{pr:chariss} allows one to prove a host of other invariance results in the same way:

\begin{lemma}\label{le:fullrelsetprops}
Let $\calg$ be a set of proper special subgroups of $A_\G$.
\begin{enumerate}
\item 
If $A_\Delta$ is a $\calg$-component of $\G$ containing at least two vertices, then $A_\Delta$ is invariant under $\out^0(A_\Gamma;\calg)$.
\item 
If $A_\Delta$ and $A_\Lambda$ are preserved by $\out^0(A_\G;\calg)$ and there exists $v\in \Delta\cap \Lambda$ with $\st(v)\subset \Delta\cup\Lambda$, then $A_{\Delta\cup\Lambda}$ is is preserved by $\out^0(A_\G;\calg)$.
\item 
If $A_\Delta$ and $A_\Lambda$ are preserved by $\out^0(A_\G;\calg)$, the intersection $\Delta\cap\Lambda\neq \varnothing$ is nonempty, and $\Lambda$ is a union of connected components of $\Gamma$, then $A_{\Delta\cup\Lambda}$ is preserved by $\out^0(A_\G;\calg)$.
\end{enumerate}
\end{lemma}
\begin{proof}
For each point, one shows upwards closure under $\leq_\calg$ and that no outside vertex $\calg$--star-separates. 

For part (1), if $v \in \Delta$ and $v \leq_\calg w$, then as $\Delta$ contains at least two vertices, there is some $u \in \Delta$ that is $\calg$-adjacent to $v$. This implies that $w$ is $\calg$-adjacent to $u$ (as $v \leq_\calg w$), so $w \in \Delta$ as $\Delta$ is closed under $\calg$--adjacency. If $x$ lies outside of $\Delta$, then $\Delta$ is $\calg^x$-connected (any element of $\calg$ contained in $A_\Delta$ is an element of $\calg^x$), so that $\Delta$ is itself a $\calg^x$-component of $\G - \st(x)$.

For part (2), as both $\Delta$ and $\Lambda$ are upwards-closed under $\leq_\calg$, so is the union $\Delta \cup \Lambda$. Let $v$ be a vertex with $v\in \Delta\cap \Lambda$ and $\st(v)\subset \Delta\cup\Lambda$. If $x \not \in \Lambda \cup \Delta$, then $x$ is not adjacent to $v$. Let $K$ be the $\calg^x$-component of $\G - \st(x)$ containing $v$. As both $A_\Delta$ and $A_\Lambda$ are preserved by $\out^0(A_\G;\calg)$ and contain $v$, they intersect $\G - \st(x)$ in a subset of $K$. Hence $\Delta \cup \Lambda$ intersects this unique $\calg^x$-component of $\G -\st(x)$, also.

The last point follows from the preceding point, since if $\Lambda$ is a union of connected components of $\G$ and $v\in \Lambda$, then $\st(v)\subset \Lambda$ also.
\end{proof}

If $\calg$ is a set of special subgroups of $A_\G$ and $\Theta$ is a subgraph of $\G$, we let $N_\calg(\Theta)$ be the union of $\Theta$ and its $\calg$-adjacent vertices.

\begin{proposition}\label{pr:peripheralinvariant}
Suppose $\calg$ is a set of proper special subgroups of $A_\G$. Let $A_\Delta \in \calg$ and let $\Theta$ be a $\calg$-connected subgraph of $\G - \Delta$. Let \[ \Lambda=N_\calg(\Theta) \cap \Delta \] be the subgraph of $\Delta$ spanned by vertices that are $\calg$-adjacent to some element of $\Theta$. Then $\out^0(A_\G;\calg)$ preserves $A_\Lambda$.
\end{proposition}

\begin{proof}
We will show that $\Lambda$ is upwards closed under $\leq_\calg$ and it not $\calg$--star-separated by an outside vertex (Proposition~\ref{pr:chariss}). First suppose that $v,w\in V(\G)$ with $v\in \Lambda$ and $v\leq_\calg w$. Then $w \in \Delta$ as $A_\Delta \in \calg$. As $v \in \Lambda$ there is a vertex $u\in \Theta$ such that $v$ is $\calg$-adjacent to $u$.
If $v$ is connected to $u$ by an edge, then so is $w$, since $v\leq_\G w$. Otherwise there is $A_\Xi\in \calg$ with $u,v\in \Xi$. Then $w\in \Xi$ also (since $v\leq_\calg w$).
In either case, $w$ is in $\Delta$ and is $\calg$-adjacent to the vertex $u \in \Theta$, and therefore $w\in \Lambda$.

It remains to show that if $v \not \in \Lambda$, the graph $\Lambda$ intersects at most one $\calg^{v}$-component of $\G - \st(v)$. If $v \not \in \Delta$, then as $A_\Delta \in \calg$, the graph $\Delta$ intersects at most one $\calg^{v}$-component of $\G - \st(v)$, so the same is true for $\Lambda$. We may therefore assume that $v \in \Delta$. Let $u$ and $w$ be any two vertices in $\Lambda - \st(v)$. By the definition of $\Lambda$, the vertices $u$ and $w$ are $\calg$-adjacent to vertices in $\Theta$.
Since $\Theta$ is $\calg$-connected, there is a $\calg$-path $p$ from $u$ to $w$ through $\Theta$ (every vertex of $p$ is in $\Theta$ except for the endpoints, which are $u$ and $w$). Since $v \in \Delta$ but $v \not \in \Lambda$, the vertex $v$ is not $\calg$-adjacent to any vertex in the interior of $p$. This implies that this path is in $\G - \st(v)$, and furthermore any $\calg$-adjacent vertices on the path are also $\calg^{v}$-adjacent, and $p$ is a $\calg^{v}$-path. Hence $u$ and $w$ are in the same $\calg^{v}$-component of $\G - \st(v)$. As $u$ and $w$ were arbitrary, $\Lambda$ intersects at most one $\calg^{v}$-component of $\G-\st(v)$.
\end{proof}

 \begin{corollary}\label{co:peripheralinvariant}
 Let $\calg$ be a set of proper special subgroups of $A_\G$, let $A_\Delta \in \calg$ and let $\calh \subset \calg$ be a subset of $\calg$. Let $\Theta$ be a $\calh$-connected subset of $\G - \Delta$. If \[ \Lambda=N_\calh(\Theta) \cap \Delta ,\] then $A_\Lambda$ is invariant under $\out^0(A_\G;\calg)$.
 \end{corollary}

 \begin{proof}
 Apply the above proposition with $\calg'=\calh \cup \{A_\Delta\}$. The result follows as $N_{\calg'}(\Theta)\cap\Delta=N_\calh(\Theta)\cap \Delta$ and a subgroup invariant under $\out^0(A_\G,\calg')$ is invariant under $\out^0(A_\G;\calg)$, also.
 \end{proof}

\begin{corollary}\label{co:peripheralinvariant2}
Let $A_\Delta \in \calg$ and let $x$ be a vertex with $x \not \in \Delta$. Then
$ A_{\lk(x) \cap \Delta }$
is preserved by $\out^0(A_\G;\calg)$.
\end{corollary}
\begin{proof}
Apply Corollary~\ref{co:peripheralinvariant} with $\calh=\emptyset$ and $\Theta=\{x\}$.
\end{proof}
 
\section{Restriction homomorphisms} \label{s:restriction}
In this section we describe the image and kernel of a restriction homomorphism on $\outgh$. For our application to finiteness properties, we also show restriction homomorphisms behave similarly when one passes to a principal congruence subgroup.

\subsection{The induced peripheral structure}

If $\calg$ is a collection of special subgroups of $A_\G$ and $\Delta$ is a subgraph of $\G$, then we define the \emph{induced peripheral structure} to be the set \[ \calg_\Delta = \{ A_{\Delta \cap \Lambda} \colon A_\Lambda \in \calg \text{ and } \Delta\cap \Lambda \neq\Delta \} \] We say that $\calg$ is \emph{saturated} if $\calg$ contains every proper special subgroup that is invariant under $\out^0(A_\G;\calg)$. We say that $\calg$ is \emph{saturated with respect to $(\calg,\calh)$} if $\calg$ contains every proper special subgroup of $A_\G$ that is invariant under $\outgh$ (this is stronger than just being saturated). If $\calg$ is saturated and $A_\Delta \in \calg$, then as the intersection of two preserved subgroups is also preserved (Corollary~\ref{co:preserveintersections}) we have $\calg_\Delta \subset \calg$. 

The induced peripheral structure determines a partial ordering $\leqgd$ of the vertices of $\Delta$, such that $u\leqgd v$  if and only if $\lk_\Delta(u) \subset \st_\Delta(v)$, and any element of $\calg_\Delta$ containing $u$ also contains $v$. We take time in this section to prove some technical results about the behavior of the induced peripheral structure. For motivation one may want to skip ahead to the proof of Theorem~\ref{th:restriction}, in which these results are essential.

\begin{proposition}\label{pr:inducedorder}
Suppose that $A_\Delta \in \calg$, and that $\calg$ is saturated. If $u \in \Delta$, then $u \leq_\calg v$ if and only if $v \in \Delta$ and $u \leqgd v$.
\end{proposition}

\begin{proof}
The forward direction of this proposition ($u\leq_\calg v$ implies that $v \in A_\Delta$ and $u \leqgd{} v$) follows from the definitions of $\leq_\calg$, $\leqgd$, and $\calg_\Delta$. 
For the converse, if $u \leqgd v$ then any special subgroup $A_\Theta \in \calg$ containing $u$ must also contain $v$ (as either $\Theta \cap \Delta = \Delta$ or $A_{\Theta \cap \Delta} \in \calg_\Delta$). 
Although $\lk_\Delta(u) \subset \st_\Delta(v)$, to finish the proof we must take care to show that $\lk_\G(u) \subset \st_\G(v)$. Let $x \in \lk_\G(u)$. If $x \in \Delta$, then $x \in \st_\G(v)$ as $\lk_\Delta(u) \subset \st_\Delta(v)$. This leaves the case when $x \not \in \Delta$. By Corollary~\ref{co:peripheralinvariant2}, the group $A_{\lk(x)\cap\Delta}$ is invariant under $\out^0(A_\G;\calg)$. Hence $A_{\lk(x) \cap \Delta} \in \calg$ as $\calg$ is saturated (therefore $A_{\lk(x) \cap \Delta} \in \calg_\Delta$, also). As $u \leqgd v$, and $u \in \lk(x) \cap \Delta$, this implies that $v \in \lk(x) \cap \Delta$, so that $x \in \st_\G(v)$.
\end{proof}

\begin{proposition}\label{pr:anchoredinvariance}
Let $A_\Delta \in \calg$ and suppose that $\calg$ is saturated. Let $v$ be a vertex of $\Delta$ and let \[ A^\Delta_{\geq v} = \genby{\{ w \in \Delta \colon v\leq_\Delta w\}}. \]
Then $A^\Delta_{\geq v} \in \calg$.
\end{proposition}

\begin{proof}
By Proposition~\ref{pr:chariss}, it is enough to show that $A^\Delta_{\geq v}$ is upwards-closed under $\leq_\calg$ and not $\calg$--star-separated by any outside vertex. Suppose that $u \in A^\Delta_{\geq v}$ and $u \leq_\calg w$ for some vertex $w$. Then $w \in \Delta$ and ${u \leqgd w}$ by Proposition~\ref{pr:inducedorder}. 
Hence $v \leq_\Delta u \leq_\Delta w$ so $w$ is also in $A^\Delta_{\geq v}$.  This shows that the vertices in $A^\Delta_{\geq v}$ are upwards-closed with respect to $\leq_\calg$. Now suppose that $x \not \in A^\Delta_{\geq v}$. If $x \not \in \Delta$ then $A_\Delta \in \calg^x$ and the graph $\Delta$ intersects at most one $\calg^{x}$-component of $\G -\st(x)$. Therefore the same is true for $A^\Delta_{\geq v}$. Otherwise $x \in \Delta$ but $v \not\leq_\Delta x$. Then there exists $u \in \lk_\Delta (v)$ that is not contained in $\st_\Delta (x)$. In this case, every vertex in $A^\Delta_{\geq v}$ is adjacent to $u$ so $A^\Delta_{\geq v}$ is contained in the union of $\st(x)$ and the connected component of $\G - \st(x)$ containing $u$. Therefore $A^\Delta_{\geq v}$ is not $\calg$--star-separated by any outside vertex.
\end{proof}

\begin{proposition}\label{pr:inducedcomponents}
Let $A_\Delta \in \calg$ and suppose that $\calg$ is saturated. Let $u$, $v$ and $x$ be vertices in $\Delta$. Then $u$ and $v$ are in the same $\calg^{x}$-component of $\G - \st(x)$ if and only if they are in the same $\calg_\Delta^{x}$-component of $\Delta - \st(x)$. 
\end{proposition}

\begin{proof}
As $\calg$ is saturated $\calg_\Delta$ can be viewed as a subset of $\calg$. Hence if $u$ and $v$ are connected by a $\calg_\Delta^{x}$-path in $\Delta - \st(x)$, then they are connected by a $\calg^{x}$-path in $\G - \st(x)$. The tricky part is therefore the converse: let $p$ be a $\calg^{x}$-path in $\G - \st(x)$ from $u$ to $v$. We want to replace $p$ with a $\calg^x$-path lying entirely in $\Delta - \st(x)$. Let us first restrict to the special case where the interior of $p$ lies in $\Gamma - \Delta$. Let $\Theta$ be the subgraph spanned by these interior points, and let \[ \Lambda = N_{\calg^{x}}(\Theta) \cap \Delta. \] Then $A_\Lambda$ is a subgroup of $A_\Delta$ which is disjoint from $x$ and contains $u$ and $v$. Furthermore, $A_\Lambda$ is invariant under $\out^0(A_\G;\calg)$ by an application of Corollary~\ref{co:peripheralinvariant} with $\calh=\calg^x$. As $\calg$ is saturated, $A_\Lambda \in \calg^{x}$ and the vertices $u$ and $v$ are $\calg^{x}$-adjacent.

In general, the path $p$ can be written as a concatenation \[p_1 \cdot l_1 \cdot p_2 \cdot l_2 \cdots p_k \cdot l_k \cdot p_{k+1}, \] where each $p_i$ is contained in $\Delta$ and each $l_i$ in $\G -\Delta$. Each $p_i$ is a $\calg_\Delta^{x}$ path, and by the previous paragraph the terminal vertex of each path $p_i$ is $\calg_\Delta^{x}$-adjacent to the initial vertex of the path $p_{i+1}$  (so the $l_i$ pieces are unnecessary).
\end{proof}

This gives a correspondence between $\calg^x$-components of $\G - \st_\Gamma(x)$ and the $\calg^x_\Delta$-components of $\Delta - \st_\Delta(x)$:

\begin{corollary}\label{c:inducedcomponents}
Let $\calg$ be saturated and let $x \in \Delta$. Then $K$ is a $\calg_\Delta^{x}$-component of $\Delta - \st_\Delta(x)$ if and only if there exists a $\calg^{x}$-component $L$ of $\G - \st_\G(x)$ such that $K=L \cap \Delta$.
\end{corollary}

\subsection{Proof of Theorem~\ref{th:restriction}}

We are now able to finish Theorem \ref{th:restriction}, which says that if $\calg$ is saturated with respect to the pair $(\calg,\calh)$, then the restriction map $R_\Delta$ applied to $\outgh$ fits in an exact sequence: \[
\begin{split}
1\to \out^0(A_\G;\calg,(\calh\cup\{A_{\Delta}\})^t) \to  &\out^0(A_\G;\calg,\calh^t) \\
& \stackrel{R_{\Delta}}{\longrightarrow} \out^0(A_{\Delta};\calg_{\Delta},\calh_{\Delta}^t)\to 1.
\end{split}
\]

\begin{proof}[Proof of Theorem~\ref{th:restriction}]
An automorphism lies in the kernel of $R_\Delta$ if and only if it restricts to an inner automorphism on $A_\Delta$, hence \[ \ker R_\Delta \cong \out^0(A_\G;\calg,(\calh\cup \{A_\Delta\})^t). \] We break up describing the image into 3 steps.\vspace{0.1cm}

\textbf{Step 1. $\im R_\Delta \subset \out(A_\Delta;\calg_\Delta,\calh_\Delta^t)$.}

\vspace{0.1cm}

Each $A_\Lambda \in \calg_\Delta$ is preserved by $\out^0(A_\G;\calg,\calh^t)$ as intersections of preserved subgroups are also preserved (Corollary~\ref{co:preserveintersections}). Similarly, if $A_\Lambda \in \calh_\Delta$ there exists a subgraph $\Lambda'$ containing $\Lambda$ such that $A_{\Lambda'} \in \calh$. Hence elements of $\out^0(A_\G;\calg,\calh^t)$ restrict to an inner automorphism on $A_{\Lambda}$, also.

\vspace{0.1cm}

\textbf{Step 2. $\im R_\Delta \subset \out^0(A_\Delta)$}

\vspace{0.1cm}

The image lies in $\out^0(A_\Delta)$ if and only if the image preserves each subgroup of the form $A^\Delta_{\geq v}$ (Proposition~\ref{pr:out0classification}). As $\calg$ is saturated, each $A^\Delta_{\geq v}$ is invariant under $\outgh$ (Proposition~\ref{pr:anchoredinvariance}), so is also invariant after restricting to $A_\Delta$.

\vspace{0.1cm}

\textbf{Step 3. $\out^0(A_{\Delta};\calg_{\Delta},\calh_{\Delta}^t) \subset \im R_\Delta$} 

\vspace{0.1cm}

We show that every inversion, transvection, and extended partial conjugation in $\out^0(A_{\Delta};\calg_{\Delta},\calh_{\Delta}^t)$ is the image of an automorphism of the same type under $R_\Delta$. As $\out^0(A_{\Delta};\calg_{\Delta},\calh_{\Delta}^t)$ is generated by such elements (Theorem~\ref{th:generators}), this is enough to finish the proof. 

As $\calg$ is saturated with respect to the pair $(\calg,\calh)$ we have $P(\calh) \subset \calg$ and $P(\calh_\Delta) \subset \calg_\Delta$ (Lemma~\ref{le:calhsubsets}). We may therefore make liberal use of Lemma~\ref{pr:gensrelgh}, which describes when such generators lie in $\outgh$ and $\out^0(A_{\Delta};\calg_{\Delta},\calh_{\Delta}^t)$. 

For inversions, if $[\iota_v] \in \out^0(A_\Delta;\calg_\Delta,\calh_\Delta^t)$ then $v \not \in A_\Lambda$ for any subgroup $A_\Lambda \in \calh_\Delta$. Hence $v \not \in A_{\Lambda'}$ for any subgroup $A_{\Lambda'} \in \calh$ and $[\iota_v]$ is the image of an inversion of the same name in $\out^0(A_\G;\calg,\calh^t)$. 

If $[\rho^w_v]$ is a transvection in $\out^0(A_\Delta;\calg_\Delta,\calh_\Delta^t)$ then $v \leqgd w$. This implies that $v \leq_\calg w$ by Proposition~\ref{pr:inducedorder}, and $[\rho^w_v]$ is the image of an element of the same name in $\out^0(A_\G;\calg,\calh^t)$. 

Finally, let $x \in \Delta$ and let $K$ be a $\calg^x_\Delta$-component of $\Delta - \st(x)$. By Corollary~\ref{c:inducedcomponents}, there exists a $\calg^{x}$-component $L$ of $\G -\st_\Gamma(v)$ such that $L \cap \Delta = K$. Then the extended partial conjugation $[\pi_{L}^x]$ maps to $[\pi_K^x]$ under $R_\Delta$. As any extended partial conjugation in $\out^0(A_{\Delta};\calg_{\Delta},\calh_{\Delta}^t)$ is a product of such elements, they are all in the image of $R_\Delta$.
\end{proof}

We take a moment to improve Theorem~\ref{th:restriction} for the purpose of computing examples.
Suppose we are given $A_\G$, $\calg$, and $\calh$, with $\calg$ not saturated.
Suppose $A_\Delta\in\calg$, and we want to compute the image of $R_\Delta$ on $\outgh$ by hand.
Certainly we could compute the saturation of $\calg$ with respect to $(\calg,\calh)$ by exhaustively checking subgraphs of $\G$.
However, this is unnecessary, since it turns out that our hypotheses in Theorem~\ref{th:restriction} are too strong.
We improve the hypotheses in the following proposition.
This is only used for computing examples (see Section~\ref{s:examples}), and not in the proofs of our main theorems.

\begin{proposition}\label{pr:easier_way_to_find_peripheral_structure}
Let $A_\G$, $\calg$, and $\calh$ be as above, with $\calg$ possibly not saturated.
If necessary, enlarge $\calg$ so that $\calh\subset\calg$, and so that for every $A_\Delta\in\calh$ and every $v\in\Delta$, we have $A_{\Delta-\{v\}}\in\calg$.
Let $A_\Delta\in\calg$, and 
define a collection $\calp_\Delta$ of special subgroups, where  $A_\Lambda \in \mathcal{P}_\Delta$ if and only if at least one of the following holds:
\begin{itemize}
 \item $\Lambda= {\lk(x)\cap \Delta}$ for some $x \not\in \Delta$.
 \item $\Lambda= N_{\calg^x}(\Theta)\cap \Delta$, where $\Theta$ is maximal among $\calg^x$-connected subgraphs of $\G-\st(x)$ that are disjoint from $\Delta$, for some $x\in\Delta$.
\end{itemize}
Then 
\[ R_\Delta \colon \outgh \to \out^0(A_\Delta;\calg_\Delta\cup\calp_\Delta,\calh_\Delta^t) \]
is surjective.
\end{proposition}
\begin{proof}
The subgroups we are adding to $\calg$ are invariant by Lemma~\ref{le:calhsubsets}, so enlarging $\calg$ does not change $\outgh$.
The subgroups in $\calp_\Delta$ are invariant by Corollaries~\ref{co:peripheralinvariant} and~\ref{co:peripheralinvariant2}, so $\outgh$ and $\out^0(A_\G;\calg\cup\calp_\Delta,\calh^t)$ are equal.
Then $R_\Delta$ is well defined.
To prove surjectivity, we review the proof of step~3 of the proof of Theorem~\ref{th:restriction}.
This amounts to explaining why Laurence generators from $\out^0(A_\Delta;\calg_\Delta\cup\calp_\Delta,\calh_\Delta^t)$ lift to $\outgh$.
As explained in Remark~\ref{re:notPh}, Proposition~\ref{pr:gensrelgh} goes through for both these groups, even though $\calg$ and $\calg_\Delta\cup\calp_\Delta$ are not saturated.
So we may use the characterization of generators in Proposition~\ref{pr:gensrelgh}.

It is obvious that inversions lift.
Lifting transvections works because Proposition~\ref{pr:inducedorder} goes through, because we include subgroups of the form $A_{\lk(x)\cap\Delta}$ for $x\notin\Delta$.
Lifting partial conjugations works because Proposition~\ref{pr:inducedcomponents} goes through, because we include subgroups of the form $\genby{N_{\calg^x}(\Theta)\cap\Delta}$.
\end{proof}

\subsection{A version of Theorem~\ref{th:restriction} for principal congruence subgroups}

The action of $\aut(A_\G)$ on the abelianization of $A_\G$ induces a homomorphism $$\Psi\colon \aut(A_\G) \to \gln$$ where $n$ is the number of vertices in $\G$. We define $\ia$ to be the kernel of $\Psi$. The notation IA is short for `identity on the abelianization.' By taking entries in matrices modulo some natural number $l \geq 2$, there is a homomorphism $$\Psi^{[l]}\colon \aut(A_\G) \to \glnp$$
Let $\iap$ be the kernel of $\Psi^{[l]}$. The main objective of the next proposition is to show that $\iap$ is contained in $\auto$: the key point is that we can detect whether an automorphism lies in $\aut^0(A_\G)$ by its action on $H_1(A_\G)$, even if we take coefficients in $\mathbb{Z} / l\mathbb{Z}$. Following the terminology used in algebraic groups, we call $\iap$ the \emph{principal congruence subgroup of level $l$}.

\begin{proposition}\label{pr:iap}
The principal congruence subgroup $\iap$ is a finite index subgroup of $\aut^0(A_\G)$. If $l \geq 3$, then $\iap$ is torsion-free.
\end{proposition}

\begin{proof}
Let $G$ be the image of $\auto$ in $\glnp$. If the vertices $v_1,\ldots,v_n$ are ordered so that $v_i \leq_\G v_j$ implies $i \leq j$ and all vertices in the same equivalence class are consecutive, then $G$ consists of block lower-triangular matrices of the form $$ M=\begin{pmatrix}M_1 & 0 & \dotsm & 0 \\
* & M_2 &  & 0 \\
\vdots & & \ddots & \vdots \\
* & * & \dots & M_{r} 
 \end{pmatrix},$$ where $r$ is the number of equivalence classes of vertices in $\G$. Each $M_i$ corresponds to the action of $\aut^0(A_\G)$ on the abelianization of the $i$th equivalence class. Each coset of $\aut^0(A_\G)$ in $\aut(A_\G)$ is represented by a graph symmetry $\alpha$ that permutes equivalence classes in $\G$ (see, for example, the proof of  Proposition~\ref{pr:out0classification}). If $\alpha$ is nontrivial, it follows that $\Psi^{[l]}(\alpha)$ is permutation matrix lying outside of $G$. Let $\phi$ be an automorphism of $A_\G$. If $\phi \not\in \auto$, then we can decompose $\phi$ as $\phi=\alpha \cdot\phi'$, where $\alpha$ is a permutation as above and $\phi' \in \auto$. As $\Psi^{[l]}(\phi') \in G$ and $\Psi^{[l]}(\alpha) \not \in G$, the image of $\phi$ is nontrivial in $\glnp$, so $\phi \not\in \iap$. When $l \geq 3$ the group $\iap$ is torsion-free as $\ia$ is torsion-free and the principal congruence subgroup of level $l$ in $\gln$ is torsion-free (see Toinet~\cite[Theorem 7.14]{MR3095717} or an alternative proof by Bregman in \cite{2016arXiv160903602B}).
\end{proof}

We define 
$\ion$ to be the image of $\ia$ in $\out(A_\G)$ and define 
$\out^{[l]}(A_\G)$ to be the image of $\iap$ in $\out(A_\G)$. 
Again by Toinet~\cite{MR3095717}, $\ion$ is torsion-free.
So we have the following corollary:

\begin{corollary}
The principal congruence subgroup $\out^{[l]}(A_\G)$ is a finite index subgroup of $\out^0(A_\G)$. If $l \geq 3$, then $\out^{[l]}(A_\G)$ is torsion-free.
\end{corollary}

We use $\outghp$ to denote the intersection of $\out(A_\G;\calg,\calh^t)$ with the group $\out^{[l]}(A_\G)$. We require the following variation of Theorem~\ref{th:restriction} in the next section.

\begin{theorem}\label{th:pcongrestriction}
Let $\calg$ and $\calh$ be sets of proper special subgroups of $A_\G$ and suppose that $\calg$ is saturated with respect to the pair $(\calg,\calh)$. Let $A_\Delta \in \calg$. Then the restriction homomorphism $R_\Delta$ applied to $\outghp$ fits in the exact sequence:
\[
\begin{split}
1\to \out^{[l]}(A_\G;\calg,(\calh\cup\{A_{\Delta}\})^t) \to  &\out^{[l]}(A_\G;\calg,\calh^t) \\
& \stackrel{R_{\Delta}}{\longrightarrow} \out^{[l]}(A_{\Delta};\calg_{\Delta},\calh_{\Delta}^t)\to 1.
\end{split}
\]
\end{theorem}

\begin{proof}
Proposition~\ref{pr:iap} tells us that $\out^{[l]}(A_\G)$ is contained in $\out^0(A_\G)$, so we may restrict the exact sequence of Theorem~\ref{th:restriction} to $\outghp$. It is clear that the kernel is then $\out^{[l]}(A_\G;\calg,(\calh\cup\{A_\Delta\})^t)$ and that the image is contained in $\out^{[l]}(A_\Delta;\calg_\Delta,\calh_\Delta^t)$ (the action on $H_1(A_\Delta;\mathbb{Z})$ is given by deleting the appropriate rows and columns from the matrix giving the action on $H_1(A_\G;\mathbb{Z})$). It remains to justify that the image contains every element of $\out^{[l]}(A_\Delta;\calg_\Delta,\calh_\Delta^t)$. To see this, note that if $\Phi \in\out^{[l]}(A_\Delta;\calg_\Delta,\calh_\Delta^t)$, we can construct an element mapping to $\Phi$ under $R_\Delta$ by writing $\Phi$ as a product in the generating set and lifting each generator to $\outgh$. The lift of each generator constructed in the proof of Theorem~\ref{th:restriction} acts by conjugation on every vertex $v$ lying outside of $\Delta$ (so trivially on the image of $v$ in $H_1(A_\G;\mathbb{Z})$), hence this lift of $\Phi$ will lie in $\out^{[l]}(A_\G)$, and $\Phi$ is in the image of $\out^{[l]}(A_\G;\calg,\calh^t)$ under $R_\Delta$.
\end{proof}

The following is used implicitly in the next section. As we work in the finite index subgroup $\outghp$ of $\outgh$, it is important to know that $\outghp$ leaves exactly the same special subgroups invariant.  

\begin{proposition}
Let $A_\Delta$ be a special subgroup of $A_\Gamma$. Then $A_\Delta$ is invariant under $\outgh$ if and only if $A_\Delta$ is invariant under $\outghp$
\end{proposition}

\begin{proof}
If $A_\Delta$ is invariant under $\outgh$ then $A_\Delta$ is invariant under the group $\outghp$. Conversely, if $A_\Delta$ is not invariant under $\outgh$, then there exists a partial conjugation or transvection $\phi$ such that $\phi(A_\Delta)$ is not conjugate to $A_\Delta$ (c.f.\ the proof of Proposition~\ref{pr:chariss}). In this case, one can check that $\phi^m(A_\Delta)$ is also not conjugate to $A_\Delta$ for all $m \geq 1$ (see for example, the proof of Lemma~\ref{le:genexercise}). As some power $\phi^m$ is in $\outghp$, it follows that $A_\Delta$ is also not invariant under $\outghp$.
\end{proof}

This is not true for subgroups admitting a trivial action, however this discrepancy can only come from a finite subgroup of inversions.

\begin{proposition}\label{pr:diff_by_inversions}
Suppose that $\outghp$ acts trivially on $A_\Delta$ but the group $\outgh$ does not. Then the image of $\outgh$ under $R_\Delta$ is a finite group generated by inversions.
\end{proposition}

\begin{proof}
If the image of $R_\Delta$ contains an extended partial conjugation or a transvection, then this element is infinite order and is the image of an element of the same type in $\outgh$. By taking an appropriate power, we attain a nontrivial element in the image of $\outghp$ under $R_\Delta$. Hence the image of $R_\Delta$ is generated by inversions.
\end{proof}

\begin{corollary}\label{co:outl_from_out0}
For every triple $(\G,\calg,\calh)$ where $\calg$ and $\calh$ are collections of special subgroups of $A_\G$ there exists a collection $\calh'$ of special subgroups such that \[ \outghp \cong \out^{[l]}(A_\G;\calg,(\calh')^t) \] and every invariant subgroup $A_\Delta$ has a trivial action under $\out^0(A_\G;\calg,(\calh')^t)$ if and only if $A_\Delta$ has a trivial action under $\out^{[l]}(A_\G;\calg,(\calh')^t)$.
\end{corollary}

\begin{proof}
Extend $\calh$ to include the groups $A_{\{x\}}$ for each inversion $\iota_x$ that appears in Proposition~\ref{pr:diff_by_inversions}.
\end{proof}

\section{Decomposing $\out^0(A_\G)$} \label{s:decomposition}

We would now like to use the exact sequences above to break up $\out^0(A_\G)$ into understandable pieces. The missing piece of the puzzle is to study what happens when every restriction map of a relative automorphism group is trivial.

\subsection{Automorphism groups with trivial restriction maps}

Having trivial restriction homomorphisms is equivalent to $\out^0(A_\G;\calg,\calh^t)$ acting trivially on every invariant special subgroup $A_\Delta$. In this case, we can assume that $\calg=\calh$ and $\calg$ is saturated with respect to $(\calg,\calg)$, so that $\outgh=\out^0(A_\G;\calg^t)$. We use this alternative formulation for the lack of nontrivial restriction maps throughout this section.
There are five different cases:

\begin{itemize}
 \item $\G$ is disconnected and $\G$ is $\calg$-disconnected.
 \item $\G$ is disconnected and $\G$ is $\calg$-connected.
 \item $\G$ is connected and the center $Z(A_\G)$ of $A_\G$ is trivial.
 \item $\G$ is connected and $Z(A_\G)$ is a proper, nontrivial subgroup.
 \item $\G$ is complete and $A_\G=\mathbb{Z}^n$ for some $n$.
\end{itemize}

In the first case $\out^0(A_\G;\calg^t)$ is a Fouxe-Rabinovitch group. In the second and third cases, $\out^0(A_\G;\calg^t)$ is a free abelian group. In the fourth case there is another simplifying exact sequence (involving a projection homomorphism rather than a restriction homomorphism). In the final case, $\out^0(A_\G;\calg^t)$ is either $\gln$ or a nice subgroup of block upper triangular matrices. The remainder of this subsection is dedicated to proofs of the above claims.

\subsubsection{$\G$ is disconnected and $\calg$-disconnected}

Here we have a Fouxe-Rabinovitch group:

\begin{proposition}\label{co:disconnected_but_G_connected_type_f}
Suppose $\calg$ is saturated with respect to $(\calg,\calg)$. 
If $\G$ is disconnected and $\calg$-disconnected, then there exists a free factor decomposition \[ \G = A_{\Delta_1} \ast A_{\Delta_2} \ast \cdots \ast A_{\Delta_k} \ast \mathbb{F}_m, \] of $A_\G$ such that \[ \out^0(A_\G;\calg^t) = \out^0(A_\G;\{A_{\Delta_i}\}_i^t). \] 
\end{proposition}

\begin{proof}
Since $\G$ is $\calg$-disconnected, there is a free splitting \[ \G = A_{\Delta_1} \ast A_{\Delta_2} \ast \cdots \ast A_{\Delta_k} \ast \mathbb{F}_m, \] where $\mathbb{F}_m$ consists of the isolated vertices in $\G$ that are not contained in any element of $\calg$ and $\Delta_1,\ldots,\Delta_k$ are the remaining $\calg$-components. 
A free factor ${\Delta_i}$ may be an isolated vertex, in which case $A_{\Delta_i}$ is an element of $\calg$. 
Otherwise $\Delta_i$ is a $\calg$-component containing at least two vertices, and $A_{\Delta_i}$ is invariant under $\out^0(A_\G;\calg^t)$ by part (1) of Lemma~\ref{le:fullrelsetprops}. 
Hence $\out^0(A_\G;\calg^t)$ is contained in the group $\out^0(A_\G;\{A_{\Delta_i}\}_i^t)$.
 Conversely, as each element of $\calg$ is contained in some $A_{\Delta_i}$, the group $\out^0(A_\G;\{A_{\Delta_i}\}_i^t)$ is contained in $\out^0(A_\G;\calg^t)$. 
\end{proof}

%

\subsubsection{$\G$ is disconnected and $\calg$-connected}

We now look at the case where $\G$ is disconnected, but any two points in $\G$ are connected by a $\calg$-path. In this situation, the group $\out^0(A_\G;\calg^t)$ can contain nontrivial partial conjugations, but all such partial conjugations commute. 

\begin{figure}[ht]
\includegraphics{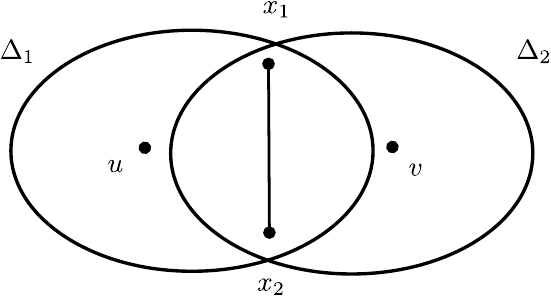} 
\caption{The group $\out^0(A_\G;\{A_{\Delta_1},A_{\Delta_2}\}^t)$ is a free abelian group with basis $\pi^{x_1}_v$ and $\pi^{x_2}_v$.}
\label{fig:pce}
\end{figure}

\begin{proposition} \label{co:disconnected_G_disconnected_type_F}
Suppose $\calg$ is saturated with respect to $(\calg,\calg)$ and that $\G$ is disconnected but $\calg$-connected. 
Then $\out^0(A_\G;\calg^t)$ is a finite rank free abelian group.
\end{proposition}

The proof of Proposition~\ref{co:disconnected_G_disconnected_type_F} is a little technical, and on first reading one may benefit from a short study of the example in Figure~\ref{fig:pce} before skipping ahead to Section~\ref{ss:case3}. For the full proof, we first need the following lemma:

\begin{lemma} \label{le:graphsetup}
Suppose that $\G$ is disconnected but $\calg$-connected. Furthermore, suppose that $\calg$ is saturated with respect to $(\calg,\calg)$. 
 \begin{enumerate}
  \item If $A_\Delta \in \calg$ and $\Lambda \subset \Delta$ then $A_\Lambda \in \calg$.
  \item Every vertex in $\G$ is contained in some element of $\calg$.
  \item For all vertices $v \in \G$, we have $A_{\{v\}} \in \calg$.
  \item If $\Delta$ is a $\calg$-connected union of connected components of $\G$ then $A_\Delta \in \calg$.
  \item There exists a connected component $\Delta$ of $\G$ such that the complement $\Lambda= \Gamma - \Delta$ is $\calg$-connected.
  \item If $\Delta$ and $\Lambda$ are as in point (5), then both $A_\Delta$ and $A_\Lambda$ are in $\calg$.
  \end{enumerate}
\end{lemma}

\begin{proof}
Item (1) follows from the fact that $\out^0(A_\G;\calg^t)$ acts trivially on every invariant subgroup $A_\Delta$, so that every special subgroup $A_\Lambda \subset A_\Delta$ must also be invariant.

For item (2), if $v$ is an isolated vertex in $\G$, then $v$ is contained in some element of $\calg$ as $\G$ is $\calg$-connected.
Otherwise, $v$ is contained in some connected component $\Delta$ with at least two vertices.
As $A_\Delta$ is invariant under the whole of $\out^0(A_\G)$, the group $A_\Delta \in \calg$, also.
Part (3) follows immediately from (1) and (2).

Part (3) implies that the $\calg$--ordering on the vertices is trivial (so $\out^0(A_\G;\calg^t)$ contains no transvections), so to show that a subgroup $A_\Delta$ is preserved, it is enough to show that $A_\Delta$ is not $\calg$--star-separated by an outside vertex.
For (4), suppose that $\Delta$ is a $\calg$-connected union of connected components of $\G$, and let $v\in \G-\Delta$.
If $A_\Lambda\in\calg$, then $A_{\Lambda-\{v\}}$ is in $\calg$ by (1), so every $\calg$-path in $\G-\st(v)$ is also a $\calg^v$-path.
So $\Delta$ is $\calg^v$-connected, since we assumed it was $\calg$-connected.
So it intersects only one $\calg^v$-component, and therefore $A_\Delta$ is invariant.
This completes (4).

For part (5), let $X$ be the graph of connected components of $\G$, with an edge between components $C_1$ and $C_2$ if there exists $A_\Delta \in \calg$ intersecting both $C_1$ and $C_2$. 
The graph $X$ is connected as $\G$ is $\calg$-connected, and there exists a connected component $\Delta$ that is not a cut vertex of $X$ (every finite connected graph that is not a point contains at least two vertices that are not cut vertices). 
Then $\Lambda=\Gamma-\Delta$ is $\calg$-connected. The final point (6) follows from (4).
\end{proof}

\begin{proof}[Proof of Proposition~\ref{co:disconnected_G_disconnected_type_F}]
By Lemma~\ref{le:graphsetup}, for every vertex $v$, the cyclic subgroup $\genby{v}$ is an element of $\calg$. 
As each $\genby{v}$ is invariant under $\out^0(A_\G;\calg^t)$, this group contains no transvections, and as the restriction to each $\genby{v}$ is trivial, $\out^0(A_\G;\calg^t)$ contains no inversions. 
Therefore $\out^0(A_\G;\calg^t)$ is generated by extended partial conjugations. 
As in Part (5) of Lemma~\ref{le:graphsetup}, we pick a connected component $\Delta$ such that the complement $\Lambda$ is $\calg$-connected. 
If both $\Delta$ and $\Lambda$ are single vertices, then $A_\G \cong \mathbb{F}_2$ and all partial conjugations are inner automorphisms, so we may assume our graph has at least three vertices. 
By possibly exchanging the roles of $\Delta$ and $\Lambda$ we may assume that either both $\Lambda$ and $\Delta$ have more than two vertices, or that $\Delta$ consists of a single isolated vertex.

\textbf{Case 1: both $\Delta$ and $\Lambda$ contain at least two vertices.}

As $\G$ is $\calg$-connected, there exist vertices $v \in \Delta$ and $w \in \Lambda$ that are contained in a common element of $\calg$. 
By Part (1) of Lemma~\ref{le:graphsetup}, this implies that the special subgroup generated by $\{v,w\}$ is also in $\calg$. 
Let $\Theta_1=\Delta \cup \{w\}$ and let $\Theta_2 = \Lambda \cup \{v\}$. 
Then $A_{\Theta_1}$ is generated by two invariant subgroups $A_\Delta$ and $A_{\{v,w\}}$ which have nonempty intersection.
As $\Delta$ is a union of connected components, $A_{\Theta_1}$ is invariant under $\out^0(A_\G;\calg^t)$ by part (3) of Lemma~\ref{le:fullrelsetprops}. 
The same reasoning shows that $A_{\Theta_2} \in \calg$, also. 
Let $[\phi] \in \out^0(A_\G;\calg^t)$ and choose $\phi$ to be a representative of $[\phi]$ restricting to the identity on $A_{\Theta_1}$. 
Then $\phi$ acts by conjugation by an element $g \in A_\G$ on $A_{\Theta_2}$. 
As $v$ and $w$ are in the intersection of $\Theta_1$ and $\Theta_2$, it follows that 
\begin{align*} \phi(v)&=v=gvg^{-1} \\ \phi(w)&=w=gwg^{-1}, \end{align*} 
so $g$ centralizes both $v$ and $w$. 
However, as $v$ and $w$ are in distinct connected components of $\G$, their common centralizer is trivial and $g$ is the identity element. 
Hence $\phi$ is trivial (as $\G = \Theta_1 \cup \Theta_2$), and as $[\phi]$ was arbitrary, $\out^0(A_\G;\calg^t)$ is trivial also.

\textbf{Case 2: $\Delta$ is a single vertex $v$.}

Suppose $\Delta = \{ v\}$ and let $x$ be any vertex of $\G$ that $\calg$--star-separates $\G$ (in other words, $x$ is an acting letter of a nontrivial extended partial conjugation in $\out^0(A_\G;\calg^t)$). 
Note that as any extended partial conjugation is inner when restricted to both $\Delta$ and $\Lambda$, it follows that $\G - \st(x)$ has exactly two $\calg^{x}$-components, one consisting of the vertex $v$, and one consisting of $\Lambda - \st(x)$. 
Then any vertex that is $\calg$-adjacent to $v$ must be contained in $\st(x)$ (such vertices exist as $\G$ is $\calg$-connected). 
Let $\Theta$ be the subgraph of $\G$ spanned by such vertices $x$. 
Then $\Theta \subset \Lambda$ as $v$ does not $\calg$--star-separate. 

Firstly, we note that it is not possible that $\Theta = \Lambda$. 
Otherwise, pick $u \in \Lambda$ that is $\calg$-adjacent to $v$. Then by our above observation, $u \in \st(x)$ for each $x \in \Theta$, so as $\Theta=\Lambda$ the vertex $u$ is central in $A_\Lambda$. In this case $\G- \st(u)$ is the single vertex $v$, and $u$ does not star-separate, which is a contradiction.

If $\Theta = \emptyset$ then $\out^0(A_\G;\calg^t)$ is trivial. 

If $\Theta \neq \Lambda$, then $\Theta \cup \{v\}$ is a proper subgraph of $\G$ which is $\out^0(A_\G;\calg^t)$ invariant as $\out^0(A_\G;\calg^t)$  contains no transvections and all vertices which are acting letters in nontrivial extended partial conjugations are contained in $\Theta$. 
Then for each $x \in \Theta$, the partial conjugation $\pi_v^x \in \out^0(A_\G;\calg^t)$, and $\pi_v^x$ is inner restricted to the subgroup generated by $\Theta \cup \{v\}$. 
Hence $x$ commutes with every other vertex in $\Theta$. 
As $x$ was arbitrary, it follows that $\Theta$ is a clique and $\out^0(A_\G;\calg^t)$ is a free abelian group generated by the partial conjugations of the form $\pi^x_v$.   
\end{proof}


\subsubsection{$\G$ is connected and $Z(A_\G)$ is trivial}\label{ss:case3}

In this case we attain the same description as in the previous section using the work of Charney and Vogtmann \cite{CVSubQuot}.

\begin{proposition}  \label{co:connected_trivial_centre_type_f}
Suppose $\calg$ is saturated with respect to $(\calg,\calg)$.
Suppose that $\G$ is connected, and that the center of $A_\G$ is trivial. 
Then $\out^0(A_\G;\calg^t)$ is a finite rank free abelian group.
\end{proposition}

\begin{proof}
Let $\G_0$ be the set of maximal equivalence classes of vertices. 
In \cite{CVSubQuot}, Charney and Vogtmann show that for each $[v] \in \G_0$, the subgroup $A_{\st[v]}$ is invariant under $\out^0(A_\G)$. 
They combine the restriction maps $R_{\st[v]}$ to give an  \emph{amalgamated restriction map} \[ R \colon \out^0(A_\G) \to \prod_{[v] \in \G_0} \out(A_{\st[v]}), \] and Theorem~4.1 of \cite{CVSubQuot} shows that when $\G$ is connected, $\ker R$ is a finitely generated free abelian group. If there is no maximal equivalence class $[v]$ with $\G=\st[v]$, then each restriction map has image in the automorphism group of a proper special subgroup. In this case, as each restriction map from $\out^0(A_\G;\calg^t)$ has trivial image, the group $\out^0(A_\G;\calg^t)$ is a subgroup of $\ker R$, so is also finitely generated and free abelian. Otherwise there exists $v$ such that $[v]$ is maximal and $\st[v]=\G$. If $[v]$ is abelian, then $A_{[v]}$ is central in $A_\G$, contradicting our hypothesis, therefore we may assume $A_{[v]}$ is non-abelian. In this case $\G$ is the join of $[v]$ and $\Delta=\G-[v]$, both $A_{[v]}$ and $A_\Delta$ are invariant under $\out^0(A_\G)$, and \[\out^0(A_\G) \cong \out^0(A_\Delta) \oplus \out^0(A_{[v]}), \] where this isomorphism is given by the product of the restriction maps $R_\Delta$ and $R_{[v]}$ on $\out^0(A_\G)$. By our hypothesis, both maps are trivial on $\out^0(A_\G;\calg^t)$, so this group is trivial.
\end{proof}

\subsubsection{$\G$ is connected and $Z(A_\G)$ is nontrivial}

When $A_\G$ has a nontrivial center the graph $\G$ can be decomposed as a join $\G=Z \ast \Delta$, where $Z$ consists of all vertices $v$ such that $\st(v)=\G$. We have the following description of $\out(A_\G)$:

\begin{proposition}[\cite{CVSubQuot}, Proposition~4.4]\label{pr:nontrivial_centre_decomposition}
If $Z(A_\G)=A_{[v]}$ is nontrivial then \[ \out(A_\G) \cong T \rtimes ((\out(A_Z)) \times \out(A_\Delta)) \] where $T$ is the free abelian group generated by transvections $\rho_w^{v'}$, where $v' \in [v]$ and $w \in \Delta$. 
The map to $\out(A_Z)$ is given by the restriction map $R_Z$ and the map to $\out(A_\Delta)$ is given by the projection map $P_\Delta$. 
The subgroup $T$ is the kernel of the product map $R_Z \times P_\Delta$.
\end{proposition}

Note that we need to use the projection homomorphism $P_\Delta$ to use a reduction argument. 
The kernels of projection homomorphisms seem a lot harder to deal with than restriction homomorphisms, although in this case the description is very nice.  
If $\calg$ is a set of special subgroups and $\calg_\Delta$ is the induced peripheral structure on $A_\Delta$, then:

\begin{proposition} \label{pr:es1}
Suppose $\calg$ is saturated with respect to $(\calg,\calg)$. Suppose that $\G$ is connected, and $Z(A_\G)$ is a proper, nontrivial subgroup of $A_\G$. If $\Delta=\G-Z$, where $Z=Z(\G)$, then there is a surjective projection homomorphism
\[P_\Delta \colon \out^0(A_\G;\calg^t) \to \out^0(A_\Delta; \calg_\Delta^t),\] with kernel a finitely generated free-abelian group.
\end{proposition}

\begin{proof}
Apply Proposition~\ref{pr:nontrivial_centre_decomposition}. The restriction map $R_Z$ is trivial, so $\ker P_\Delta$ is a subgroup of $T$ and hence is a finitely generated free-abelian group (one sees it is freely generated by transvections $\rho_w^{v'}$, where $v' \in [v]$ and $w \in \Delta$ and $w\leq_\calg v'$). It remains to determine the image of $P_\Delta$. To see this, note that if $\Phi$ acts trivially on $A_\Theta$ then $P_\Delta(\Phi)$ acts trivially on $A_{\Theta-Z}\cong A_\Theta/A_{\Theta \cap Z}$. Hence the image of $P_\Delta$ is contained in $\out^0(A_\Delta; \calg_\Delta^t)$. To see that $P_\Delta$ surjects onto $\out^0(A_\Delta; \calg_\Delta^t)$, one can use an extension argument similar to the proof of Proposition~\ref{pr:generating1}; if $\phi$ is a representative of an element of $\out^0(A_\Delta; \calg_\Delta^t)$, one can extend $\phi$ to an automorphism $\tilde \phi$ of $\out^0(A_\G;\calg^t)$ by asking that $\tilde \phi(v) =v$ for all $v \in Z$ and deduce that $\tilde \phi$ represents an element of $\out^0(A_\G;\calg^t)$ mapping onto $[\phi]$ under $P_\Delta$ (this construction actually gives a splitting $\out^0(A_\Delta; \calg_\Delta^t) \to \out^0(A_\Gamma;\calg^t)$ following the proof of Proposition~\ref{pr:nontrivial_centre_decomposition}).
\end{proof}

As the action on the abelianization of $A_\Delta$ under $P_\Delta$ is obtained by deleting appropriate rows and columns from the matrix determining the action on the abelianization of $A_\G$, we have a corresponding result in the case of level $l$ congruence subgroups:

\begin{proposition} \label{pr:connected_nontrivial_centre_exact_sequence}
Suppose $\calg$ is saturated with respect to $(\calg,\calg)$. Suppose that $\G$ is connected, and $Z(A_\G)$ is a proper, nontrivial subgroup of $A_\G$. If $\Delta=\G-Z$, where $Z=Z(\G)$, then there is a surjective projection homomorphism
\[P_\Delta \colon \out^{[l]}(A_\G;\calg^t) \to \out^{[l]}(A_\Delta; \calg_\Delta^t),\] with kernel a finitely generated free-abelian group.
\end{proposition}

The situation where $A_\G=\mathbb{Z}^n$ and there are no restriction maps is easy to handle:

\begin{proposition}\label{pr:no_restriction_gl}
Suppose that $A_\G=\mathbb{Z}^n$ and $\calg$ is a set of special subgroups that is saturated with respect to $(\calg,\calg)$. If $\out^0(A_\G;\calg^t)$ is nontrivial, then there exists $1 \leq m\leq n$ such that $\out^0(A_\G;\calg^t)$ fits in the exact sequence \[ 1\to A \to \out^0(A_\G;\calg^t) \to \glm \to 1, \] where $A$ is a finitely generated free-abelian group.
\end{proposition}

We allow for the possibility that $A$ is trivial and $m=n$.

\begin{proof}
For any vertex $v$, the subgroup $A_{\geq_\calg v}$ generated by the vertices $w$ such that $v \leq_\calg w$ is preserved by $\out^0(A_\G;\calg^t)$ (in this particular situation, star-separation is not an issue as there are no partial conjugations but one can show that this statement holds more generally). Unless $v$ is dominated by all other vertices, there is a restriction map to $A_{\geq_\calg v}$ and as restriction maps are trivial, all the automorphisms in $\out^0(A_\G;\calg^t)$ fix $v$. We may therefore assume there is an equivalence class $[v]_\calg$ containing $m$ elements where every element of this equivalence class is $\calg$-dominated by all other vertices. As $A_\G=\mathbb{Z}^n$, this is simply the elements of $\G$ not contained in any element of $\calg$. The group $\out^0(A_\G;\calg^t)$ therefore has a block lower-triangular decomposition where each element is of the form $$ M=\begin{pmatrix}I &  0 \\
 
B & C 
 \end{pmatrix},$$
 where the matrix $C$ corresponds to projecting to the action on $[v]_\calg$ and we have the identity in the top left by triviality of restriction maps. The kernel of the map projecting to the bottom right entry is a finitely generated free abelian group of rank $m(n-m)$. As all the elements of $[v]_\calg$ are disjoint from elements of $\calg$, the image of this projection map is $\glm$ (it contains all inversions and transvections between elements of $[v]_\calg$).
\end{proof}

\subsection{The decomposition theorem}

We can now prove the following result, which in particular implies Theorem~\ref{th:decomposition} from the introduction:

\begin{theorem}\label{pr:dismantleoutAGamma}
Let $A_\G$ be a RAAG and $\calg$, $\calh$ any sets of special subgroups.
Then $\out^0(A_\G;\calg,\calh^t)$ has a finite subnormal series
\[1=N_0\leq N_1 \leq \dotsm\leq N_k=\out^0(A_\G;\calg,\calh^t),\]
such that for every $i$, the group $N_{i+1}/N_i$ is isomorphic to
\begin{itemize}
\item a finitely generated free abelian group,
\item $\gl(m,\Z)$ for some $m\geq 1$, or
\item a Fouxe-Rabinovitch group $\out(A_\Delta;\calk^t)$, where $\Delta$ is a subgraph of $\G$, and $\calk$ consists of special subgroups giving a free factor decomposition of $A_\Delta$.
\end{itemize}
\end{theorem}
Note that in this last case, the group can be $\out(\mathbb{F}_m)$ for some $m$ if $\Delta$ is edgeless and $\calk=\varnothing$.
In many examples, $\gl(1,\Z)\cong \mathbb{Z}/2\mathbb{Z}$ shows up as a factor many times.

\begin{proof}
Given a triple $(\G,\calg,\calh)$, where $\calg$ and $\calh$ are sets of special subgroups of $A_\G$, we let \[ n(\G,\calg,\calh)=|V(\G)|,\] and let \[ m(\G,\calg,\calh)=2^{|V(\G)|}-r \] where $r$ is the number of special subgroups $A_\Delta$ in $A_\Gamma$ such that $\outgh$ acts trivially on $A_\Delta$. 
We use the lexicographic ordering on such pairs $(n,m)$ with the minimal element $(0,0)$. 
We fix a triple $(\G,\calg,\calh)$ of complexity $(n,m)$ and by induction assume that the result holds for $\out^0(A_\G,\calg,\calh^t)$ for all triples $(\G',\calg',\calh')$ with strictly lower complexity.

Without changing the automorphism group (or complexity) we may assume that $\calg$ is saturated for the pair $(\calg,\calh)$.  We first look at the situation where there exists $A_\Delta \in \calg$ such that the associated restriction homomorphism $R_\Delta$ on $\outgh$ has nontrivial image. In this case, Theorem~\ref{th:restriction} provides us with an exact sequence:
\[ 1 \to \out^{0}(A_\G,\calg,(\calh \cup \{A_\Delta\})^t) \to \outgh \to \out^0(A_\Delta, \calh_\Delta, \calg_\Delta) \to 1. \]
The triple $(\G,\calg,\calh\cup\{A_\Delta\})$ has strictly lower complexity than $(\G,\calg,\calh)$ as there exist automorphisms in $\outgh$ with a nontrivial action on $A_\Delta$.  The triple $(A_\Delta,\calg_\Delta,\calh_\Delta)$ has strictly lower complexity as $\Delta$ has fewer vertices. As we can find subnormal series of the left and right terms in this exact sequence that satisfy the hypothesis of the theorem, they can be combined to give the result for $\outgh$.

In the case that $\outgh$ has trivial restriction maps, we look at the five subcases studied in the previous section. In the first three cases, $\outgh$ is isomorphic to a free-abelian group or a Fouxe-Rabinovitch group. In the fourth case where $\G$ is connected and $Z(\G)$ is a proper, nontrivial subgraph, we use the exact sequence \[1 \to A \to \out^0(A_\G;\calg^t) \to \out^0(A_\Delta; \calg_\Delta^t)\to1,\] with $A$ a finitely generated free abelian group given by Proposition~\ref{pr:es1}. As the image is of lower complexity, induction gives the result here. When $A_\G$ is a free abelian group and there are no nontrivial restriction maps, the result follows from Proposition~\ref{pr:no_restriction_gl}.\end{proof}

\subsection{Building up classifying spaces}\label{ss:buildingblocks}

The main tool we need to build classifying spaces is the following result from Geoghegan's book:

\begin{theorem}[\cite{MR2365352}, Theorem 7.3.4] \label{th:F_from_rigid_actions}
If a group $G$ acts cocompactly by rigid homeomorphisms on a contractible CW complex $X$ such that the stabilizer of each cell is type $F$, then $G$ is type $F$ also.
\end{theorem}

An action is \emph{rigid} if every cell that is fixed setwise by an element is also fixed pointwise. In particular, if $X$ is the universal cover of a CW complex $Y$ with fundamental group $\pi_1(Y)=G$, then $G$ acts rigidly on $X$. We note two corollaries:

\begin{corollary}\label{co:exact_sequences_preserve_F}
Suppose that $1 \to N \to G \to Q \to 1$ is a short exact sequence such that $Q$ and $N$ are of type F. Then $G$ has type F.
\end{corollary}

\begin{proof}
Take $X$ to be the universal cover of a finite classifying space for $Q$, and use the action of $G$ on $X$ via the quotient map. The stabilizer of each cell is isomorphic to $N$.
\end{proof}

\begin{corollary}\label{co:typeFaction}
If a group $G$ acts simplicially and cocompactly on a contractible simplicial complex $X$ such that the stabilizer of each simplex is type $F$, then $G$ is type $F$ also.
\end{corollary}

\begin{proof}
The action of $G$ on the barycentric subdivision $X'$ of $X$ is rigid. Furthermore each stabilizer of a simplex in $X'$ is finite index in the stabilizer of a simplex in $X$ (so is type $F$ also).
\end{proof}

Our building blocks consist of classifying spaces for RAAGs (Salvetti complexes), symmetric spaces for arithmetic groups, and deformation spaces of trees. We will highlight the important results below. Every right-angled Artin group acts freely and cocompactly on a CAT(0) cube complex, hence:

\begin{theorem}[Charney--Davis \cite{MR1368655}]\label{th:CharneyDavis}
Every right-angled Artin group is type $F$.
\end{theorem}

Although the action of an arithmetic group $\G$ on its associated symmetric space $X$ is not always cocompact, Borel and Serre (equivariantly) extend $X$ to a space $\overline{X}$ in such a way that $\overline{X}/\G$ is a compact \emph{differentiable manifold with corners}. Triangulating this manifold gives:

\begin{theorem}[Borel--Serre, {\cite[Section~11]{MR0387495}}] 
If $\G$ is a torsion-free, arithmetic subgroup of $\gl(n,\mathbb{Q})$ then $\G$ is type $F$.
\end{theorem}

If $A_\G \cong \mathbb{Z}^n$ is free abelian and $\calg$ is a set of proper special subgroups, then $\out^0(A_\G;\calg^t)$ is an arithmetic group (fixing a basis element corresponds to setting certain matrix entries in $\gln$ to 0 or 1). For $l \geq 3$, the torsion-free subgroup $\out^{[l]}(A_\G;\calg^t)$ is finite index in $\out(A_\G;\calg^t)$, therefore:

\begin{proposition} \label{pr:borel-serre_type_f}
If $A_\G\cong \mathbb{Z}^n$ is a free abelian group, $\calg$ is a set of special subgroups of $A_\G$, and $l \geq 3$, then $\out^{[l]}(A_\G;\calg^t)$ is type F.
\end{proposition}

We now move to the free and non-abelian picture. A \emph{free factor system} for a group $G$ is a finite set of subgroups $\calg=\{G_1,\ldots, G_k\}$ of $G$ such that \[ G= G_1 \ast G_2 \ast \cdots G_k \ast \mathbb{F}_m \] for some free group $\mathbb{F}_m$ 
(this is not the same thing as a Grushko decomposition for $G$, since factors may be freely decomposable or infinite cyclic).
The group $\out(G;\calg)$ acts on a space $\spine$ defined in an analogous way to the spine of outer space. Each vertex of $\spine$ is given by a minimal, simplicial action of $G$ on a tree $T$ such that the vertex stabilizers of $T$ are exactly the conjugates of the elements of $\calg$ (this implies that the edge stabilizers are trivial). Two trees are equivalent if they are $G$--equivariantly homeomorphic. A simplex of $\spine$ is determined by a sequence $T_0,\ldots,T_k$ of such trees where $T_i$ is obtained $T_{i-1}$ by collapsing $G$-orbits of edges. The relative automorphism group $\out(G;\calg)$ admits a simplicial action on $\spine$ by precomposing the action on a tree with a representative of an outer automorphism. 

\begin{theorem}[Guirardel--Levitt, \cite{GL2}]\label{th:GL2}
Let $G$ be a group and let $\calg$ be a free factor system of $G$. Then the spine $\mathcal{X}(\calg)$ is contractible, and the group $\out(G;\calg^t)$ acts cocompactly on $\mathcal{X}(\calg)$.
\end{theorem}

The group  $\out(G;\calg^t)$ is the Fouxe-Rabinovitch group of the free factor system $\calg$. If $k+m=2$ then the spine consists of a single vertex $T$ (in other words, the whole group $\out(G;\calg)$ fixes a tree). In general, there is a projection map: 
\[ P\colon \out(G;\calg) \to \out(\mathbb{F}_m) \]
induced by the quotient map $G \to \mathbb{F}_m$ which sends every every element of a group $G_i$ to the identity element. 

\begin{proposition}[Guirardel--Levitt, \cite{GL}, Proposition 3.7]\label{pr:GLtypef}
Suppose that $G= G_1 \ast G_2 \ast \cdots G_k \ast \mathbb{F}_m$ is a free factor decomposition of $G$, with $\calg=\{G_1,G_2,\dotsc, G_k\}$.
Let $\out^1(G;\calg^t)$ be a finite index subgroup of $\out(G;\calg^t)$. 
If \begin{enumerate}
\item each group $G_i$ and its center $Z(G_i)$ has a finite classifying space, and
\item the image of $\out^1(G;\calg^t)$ in $\out(\mathbb{F}_m)$ is torsion-free,
\end{enumerate}
then $\out^1(G;\calg^t)$ has a finite classifying space.    
\end{proposition}

To prove the above, Guirardel and Levitt combine Corollary~\ref{co:typeFaction} with the cocompact action of $\out^1(G;\calg^t)$ on the spine $\calx(\calg)$. Conditions (1) and (2) ensure that the intersection of $\out^1(G;\calg^t)$ with the stabilizer of a tree $T \in \calx(\calg)$ is type $F$. 

We now specialize to the situation where 
\[ A_\G = A_{\Delta_1} \ast A_{\Delta_2} \ast \cdots \ast A_{\Delta_k} \ast \mathbb{F}_m \]
is a free factor system of a RAAG and $\calg=\{A_{\Delta_i}\}$ is a family of special subgroups. Each $\Delta_i$ is a union of connected components of $\G$, and there are $m$ isolated vertices that are not contained in any $\Delta_i$ (it may be that some $\Delta_i$ is disconnected, or is a single vertex).
If $l \geq 3$, then the image of the principal congruence subgroup $\out^{[l]}(A_\G;\calg^t)$ in $\out(\mathbb{F}_m)$ is the torsion-free group $\out^{[l]}(\mathbb{F}_m)$. 
Proposition~\ref{pr:GLtypef} and Theorem~\ref{th:CharneyDavis} imply:

\begin{proposition} \label{pr:FR_group_given_by_FF_system_is_F}
Let $A_\G = A_{\Delta_1} \ast A_{\Delta_2} \ast \cdots \ast A_{\Delta_k} \ast \mathbb{F}_m$ be a free factor decomposition of a RAAG $A_\G$, where each $A_{\Delta_i}$ is a special subgroup, and let $l\geq 3$. Then $\out^{[l]}(A_\G;(\{A_{\Delta_i}\}_{i=1}^k)^t)$ is type $F$.
\end{proposition}

\subsection{Proof of Theorem~C} \label{ss:proofofF}

Combining all the above work allows use to prove the following theorem, which in particular implies Theorem~C from the introduction:

\begin{theorem}
Let $\calg$ and $\calh$ be sets of proper special subgroups of $A_\G$. Then the principal congruence subgroup $\outghp$ is type F.
\end{theorem}

\begin{proof}
For each triple $(\G,\calg,\calh)$, we use a similar complexity $(n,m)$ as in the proof of Theorem~\ref{pr:dismantleoutAGamma},  with $n(\G,\calg,\calh)=|V(\G)|$ and $m(\G,\calg,\calh)=2^{|V(\G)|}-r$. The only difference in this proof is that we take $r$ to be the number of special subgroups $A_\Delta$ in $A_\Gamma$ such that $\outghp$ (rather than $\outgh$) acts trivially on $A_\Delta$. We proceed by induction using the lexigraphic order on complexities.

Fix a triple $(\G,\calg,\calh)$ of complexity $(n,m)$ and by induction assume that the group $\out^{[l]}(A_\G,\calg',(\calh')^t)$ has a finite classifying space for all triples $(\G',\calg',\calh')$ with strictly lower complexity. In the case that there is a nontrivial restriction map $R_\Delta$ on $\out^{[l]}(A_\G,\calg,\calh^t)$, the congruence subgroup version of the exact sequence given in Theorem~\ref{th:pcongrestriction} gives us:
\[ 1 \to \out^{[l]}(A_\G,\calg,(\calh \cup \{A_\Delta\})^t) \to \outghp \to \out^{[l]}(A_\Delta, \calh_\Delta, \calg_\Delta) \to 1. \]
The left and right sides are of lower complexity, so by induction they have finite classifying spaces. Hence $\outghp$ has a finite classifying space by Corollary~\ref{co:exact_sequences_preserve_F}. We may therefore assume that each restriction map $R_\Delta$ on the group $\outghp$ is trivial. By Corollary~\ref{co:outl_from_out0}, we can extend $\calh$ without changing $\outghp$ so that every restriction map is also trivial on $\outgh$. We go back to the cases looked at previously: 

\begin{itemize}
 \item $\G$ is disconnected and $\G$ is $\calg$-disconnected.
 \item $\G$ is disconnected and $\G$ is $\calg$-connected.
 \item $\G$ is connected and $Z(\G)$ is empty.
 \item $Z(\G)$ is nonempty, but $\G$ is not complete.
 \item $\G$ is complete.
\end{itemize}

In the first case, $\outgh$ is a Fouxe-Rabinovitch group (Proposition~\ref{co:disconnected_but_G_connected_type_f}) and the work of Guirardel--Levitt (Proposition~\ref{pr:FR_group_given_by_FF_system_is_F}) implies that $\outghp$ is type F.
In the second and third cases $\outgh$ is a finitely generated free abelian group by Proposition~\ref{co:disconnected_but_G_connected_type_f} and Proposition~\ref{co:connected_trivial_centre_type_f}, respectively.  Hence $\outghp$ is also finitely generated free abelian, and is of type F.
In the final case, since $A_\G \cong \mathbb{Z}^n$, we know $\outghp$ is an arithmetic group and type F follows from the work of Borel and Serre in Proposition~\ref{pr:borel-serre_type_f}. 

Only the fourth case remains. We can assume that $\calg$ is saturated with respect to $(\calg,\calg)$ and  $\outgh=\out^0(A_\G,\calg^t)$. Proposition~\ref{pr:connected_nontrivial_centre_exact_sequence} gives us the exact sequence
\[ 1 \to A \to \out^{[l]}(A_\G;\calg^t) \xrightarrow{P_{\Delta}}  \out^{[l]}(A_\Delta,\calg_\Delta^t) \to 1 ,\]
where $A$ is a finitely generated free abelian group. By our inductive hypothesis, the group $\out^{[l]}(A_\Delta;\calg_\Delta^t)$ is type F ($\Delta$ is a proper subgraph of $\G$), so the whole group $\out^{[l]}(A_\G;\calg^t)$ is type F also.
\end{proof}

\section{Computation and examples} \label{s:examples}

We start this section by listing a few special cases where relative (outer) automorphism groups of RAAGs have previously been considered.
After that, we give concrete examples of triples $(\G,\calg,\calh)$ that illustrate some of the phenomena in the paper.
In the final subsection, we give some new computations of virtual cohomological dimension of outer automorphism groups of RAAGs.

\subsection{Previously studied relative automorphism groups}\label{ss:previously}
We believe that the study of $\out(A_\G;\calg,\calh^t)$ in full generality is new in the present paper.
However, many special cases have been studied before.

\subsubsection{Relative linear groups}
Suppose $F$ is a subspace arrangement in $\mathbb{Q}^n$ where each subspace in $F$ is the span of a subset of the standard basis $e_1,\dotsc, e_n$.
Then the stabilizer $\gln{}_F$ is $\out(A_\G;\calg)$, where $\G$ is the complete graph on $n$ vertices, and $\calg$ is the collection of special subgroups corresponding to the subspaces in $F$.
After an appropriate reordering of the basis, every such $\gln{}_F$ is a group of block-triangular matrices.
Some familiar special cases include:
\begin{itemize}
\item The affine group, when $F=\{\genby{e_1,\dotsc,e_{n-1}}\}$.
\item The group of upper-triangular matrices in $\gln$, when 
\[F=\{\genby{e_1},\genby{e_1,e_2},\dotsc,\genby{e_1,\dotsc,e_{n-1}}\}.\]
When $n=3$, this group contains the 3-dimensional Heisenberg group as a subgroup of index $8$.
\item More generally, the subgroup of $\gln$ stabilizing a flag $F$.
\item The subgroup of $\gln$ stabilizing a direct product decomposition of $\Z^n$, when $F$ is the list of factors in the decomposition.
\end{itemize}
Of course, stabilizers of flags and subspace arrangements can be considered in much more general settings (working over other rings, without requiring special subgroups, in classical linear groups, etc.).

\subsubsection{Relative automorphism groups of free groups}
There are a few special cases that have previously been studied.
Here we specialize to $A_\G=\mathbb{F}_n=\genby{x_1,\dotsc,x_n}$, so $\G$ is the edgeless graph on $n$ vertices.

If $\calg=\{\genby{x_1},\genby{x_{2}},\dotsc,\genby{x_k}\}$ for some $k$, then $\aut(A_\G;\calg^t)$ is a \emph{partially symmetric automorphism group} of $\mathbb{F}_n$.
These groups have been studied by Bux--Charney--Vogtmann~\cite{BCV}, Day--Putman~\cite{DPBES}, Jensen--Wahl~\cite{JW}, Wade~\cite{WadeFolding} and others.
In the special case that $k=n$, this is called the \emph{pure symmetric automorphism group} or the \emph{basis-conjugating automorphism group}, and has been studied by McCool~\cite{McCool} and Collins~\cite{Collins}.

In another direction, if $\calg$ is a collection of disjoint special subgroups of $\mathbb{F}_n$, then $\calg$ is a \emph{free factor system} for $\mathbb{F}_n$, and $\out(A_\G;\calg)$ is the stabilizer of the free factor system $\calg$.
Free factor systems and their stabilizers have been considered by Bestvina--Feighn--Handel~\cite{BFH}, and Handel--Mosher~\cite{2013arXiv1302.2681H}, and others.
More generally, Handel--Mosher consider \emph{subgroup systems} and their stabilizers.
A subgroup system is a collection $\calg$ of finite-rank subgroups of $\mathbb{F}_n$, considered up to conjugacy, and the stabilizer of $\calg$ is simply $\out(A_\G;\calg)$ (in our notation).
Both free factor systems and subgroup systems are usually studied without any requirement that the subgroups be conjugate to special subgroups.

\subsubsection{Automorphisms of free products}

As noted above, Fouxe-Rabinovitch groups are examples of relative automorphism groups.
Let $G$ be a group with a finite Grushko decomposition $G=G_1*\dotsm*G_m*\mathbb{F}_k$, and let $\calg=\{G_1,\dotsc,G_m\}$.
Then $\aut(G;\calg)$ and $\aut(G;\calg^t)$ both play a role in the study of $\aut(G)$.
By giving a list of generators, Fouxe-Rabinovitch defined a group in~\cite{FR} that coincides with $\aut(G;\calg^t)$.
These groups have come up again in more recent work, including Carette~\cite{Carette}, Guirardel--Levitt~\cite{GL2,GL}, and McCullough--Miller~\cite{MM}.
Any RAAG with disconnected defining graph has a nontrivial Grushko decomposition, with special subgroups as free factors.
Therefore these Fouxe-Rabinovitch groups are relative automorphism groups of RAAGs.

\subsubsection{Other relative automorphism groups of RAAGs}
We recall that $\aut^0(A_\G)$ and $\out^0(A_\G)$ are relative (outer) automorphism groups, as shown in Proposition~\ref{pr:out0classification}.

For a general RAAG $A_\G$, set $\calg=\{\genby{x} | x\in V(\G)\}$.
Then $\aut(A_\G;\calg^t)$ is the \emph{pure symmetric automorphism group} and $\out(A_\G;\calg^t)$ is the \emph{pure symmetric outer automorphism group}.
These groups have been studied by Day--Wade~\cite{DW}, Koban--Piggott~\cite{KP}, and Toinet~\cite{ToinetPSA}.

In a slightly different direction, Duncan--Remeslennikov \cite{DuncanR3} show that the subgroup of $\aut(A_\G)$ generated by inversions and transvections is equal to the group of automorphisms preserving each $A_{\geq v}$ setwise (rather than up to conjugacy). They find a presentation for this group, which is described as $\rm{St}(\mathcal{K})$ in \cite{MR2999373,DuncanR3}.

Another example, perhaps more surprising, is that the untwisted outer automorphism group of a RAAG, as defined by Charney--Stambaugh--Vogtmann~\cite{MR3626599}, turns out to be virtually a relative outer automorphism group.
The \emph{untwisted} subgroup $U(A_\G)$ of $\out(A_\G)$ is defined to be the subgroup generated by
\begin{itemize}
\item graphic automorphisms,
\item inversions,
\item partial conjugations, and
\item non-adjacent transvections: automorphisms $\rho^v_w$ where $v\notin\lk(w)$.
\end{itemize}
So $U(A_\G)$ is generated by the Laurence generating set with the adjacent transvections taken out. 
Charney--Stambaugh--Vogtmann~\cite{MR3626599} have shown that $U(A_\G)$ has a proper action on a space, called an \emph{Outer space} for $A_\G$. 
Let $U^0(A_\G)$ be the intersection of the untwisted subgroup with $\out^0(A_\G)$. 
We show below that $U^0(A_\G)$ is a relative automorphism group. 
For any vertex $v$, let $A_{\geq v}^{{\text{NA}}}$ be the subgroup generated by $v$ and the vertices $w$ such that $v \leq w$ and $v$ is not adjacent to $w$. 
Let $\mathcal{G}_{\geq}^{{\text{NA}}}$ be the set of all proper special subgroups of the form $A_{\geq v}^{{\text{NA}}}$.

\begin{proposition}
The group $U^0(A_\G)$ is equal to $\out^0(A_\G;\mathcal{G}_{\geq}^{{\text{NA}}})$.
\end{proposition}

\begin{proof}
If $\rho_{v}^w$ is an adjacent transvection, then as $v$ is contained in $A_{\geq v}^{{\text{NA}}}$ and $w$ is not, this transvection is not in $\out^0(A_\G;\mathcal{G}_{\geq}^{{\text{NA}}})$ by Proposition~\ref{pr:gensrelgh}. As $\out^0(A_\G;\mathcal{G}_{\geq}^{{\text{NA}}})$ is generated by the inversions, transvections, and extended partial conjugations it contains (Theorem~D), and all of these elements lie in $U^0(A_\G)$, it follows that $\out^0(A_\G;\mathcal{G}_{\geq}^{{\text{NA}}})$ is contained in $U^0(A_\G)$. 
For the converse, we want to show that $\out^0(A_\G;\mathcal{G}_{\geq}^{{\text{NA}}})$ contains every other inversion, transvection, and extended partial conjugation in $\out^0(A_\G)$. 
This is clear for inversions. Now suppose that $v \leq w$ and $v$ and $w$ are non-adjacent. 
Suppose $v$ is contained in $A_{\geq u}^{{\text{NA}}}$ for some vertex $u$. If $u=w$ then clearly $w \in A_{\geq u}^{{\text{NA}}}$, also. Suppose $u \neq w$. As $\lk(u) \subset \st(v)$ and $w$ is not adjacent to $v$, the vertex $w$ is also not adjacent to $u$. As $u \leq v \leq w$, it follows that $u \leq w$, also.  Hence $w$ dominates $u$ non-adjacently, and therefore $w \in A_{\geq u}^{{\text{NA}}}$. 
As every element of $\mathcal{G}_{\geq}^{{\text{NA}}}$ containing $v$ also contains $w$, and the transvection $\rho^w_v$ is contained in $\out^0(A_\G;\mathcal{G}_{\geq}^{{\text{NA}}})$. 

Now suppose that $x$ is a vertex and $A_{\geq v}^{{\text{NA}}}$ is an element of $\mathcal{G}_{\geq}^{{\text{NA}}}$ that does not contain $x$. 
We shall show that $A_{\geq v}^{{\text{NA}}}$ intersects at most one component of $\G - \st(x)$, so is preserved by every partial conjugation $\pi^x_K$. This implies that $\out^0(A_\G;\mathcal{G}_{\geq}^{{\text{NA}}})$ contains all partial conjugations and will complete the proof. 
There are two cases. 
If $x$ does not dominate $v$ then there exists $u \in \lk(v)$ which is in $\G -\st(x)$. As $u$ is adjacent to every element of $A_{\geq v}^{{\text{NA}}}$, this group intersects only the component of $\G -\st(x)$ containing $u$. 
Otherwise, $x$ dominates $v$ adjacently. As $x \in \lk(v)$, every vertex in $A_{\geq v}^{{\text{NA}}}$ is adjacent to $x$ and the group is contained in $A_{\st(x)}$. 
In this case, every partial conjugation $\pi^x_K$ acts trivially on $A_{\geq v}^{{\text{NA}}}$.
\end{proof}

\subsection{Illustrative examples}

\subsubsection{Excluded transvections and partial conjugations}
\begin{figure}
\begin{centering}
\includegraphics{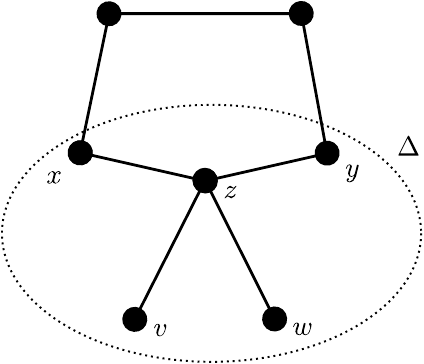}
\end{centering}\caption{An example graph for looking at the peripheral structure in the absolute case: here $A_\Delta$ is preserved by the whole of $\out^0(A_\G)$.}\label{fig:luckyjim}
\end{figure} 

\begin{example}\label{ex:luckyjim}
For the graph given in Figure~\ref{fig:luckyjim}, the subgraph $A_\Delta$ generated by $v,w,x,y,$ and $z$ is invariant under $\out^0(A_\G)$. 
This is because it is the star of a maximal equivalence class of vertices, and 
Charney--Vogtmann~\cite{CVSubQuot} showed these are invariant.
The peripheral subgroups $\mathcal{P}_\Delta$ given by Proposition~\ref{pr:easier_way_to_find_peripheral_structure} are $\{x\}$, $\{y\}$, and $\{x,y\}$. Hence the image of the restriction map $R_\Delta$ is \[ \out^0(A_\Delta;\mathcal{P}_\Delta) =\out^0(A_\Delta;\{\genby{x},\genby{y},\genby{x,y}\}).\] 
The group $A_\Delta$ is isomorphic to $\mathbb{F}_4 \oplus \mathbb{Z}$, and this relative automorphism group contains all the inversions in $\out^0(A_\Delta)$, but a proper subset of the extended partial conjugations and transvections. 

For example, the group $\out^0(A_\Delta)$ contains the partial conjugations $\pi^v_x$ and $\pi^v_y$. 
These are not in the relative automorphism group as they do not preserve $\genby{x,y}$. 
However the product $\pi^v_x\pi^v_y$ is an extended partial conjugation in $\out^0(A_\Delta;\mathcal{P}_\Delta)$.

Likewise, we can see that $\out^0(A_\Delta)$ contains the transvection $\rho^v_x$, but $\out^0(A_\G)$ and $\out^0(A_\Delta;  \calp_\Delta)$ do not.
Certainly $v\geq_\Delta x$, but $v\not\geq_\Gamma x$, and $v\not\geq_{\calp_\Delta} x$.
\end{example}

\subsubsection{Subgroups of groups in $\calh$ and amalgamated restriction maps}

There is a temptation to economize by amalgamating several restriction maps into a single map.
However, the images of amalgamated restriction maps can be difficult to describe.
This is related to the phenomenon of including groups $A_{\Delta-\{v\}}$ in $\calg$ when $A_\Delta\in\calh$, as in Proposition~\ref{pr:easier_way_to_find_peripheral_structure}.

\begin{figure}
\begin{centering}
\includegraphics{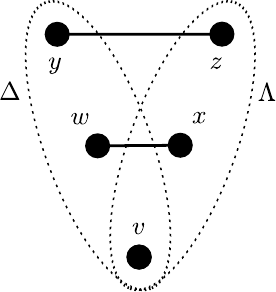}
\end{centering}\caption{An example illustrating the necessity of taking subgroups of groups in $\calh$.}\label{fig:taurus}
\end{figure}

\begin{example}\label{ex:taurus}
Let $\G$ be the graph in Figure~\ref{fig:taurus}, with $\Delta$ and $\Lambda$ the indicated subgraphs.
We set $\Xi=\Delta\cap\Lambda$, and set $\calg=\{A_\Delta,A_\Lambda\}$ and set $\calh=\{A_\Xi\}$.
The group $\outgh$ has an amalgamated restriction homomorphism
\[R\colon \outgh \to \out^0(A_\Delta;\calg_\Delta\cup \calp_\Delta,\calh_\Delta^t)\oplus \out^0(A_\Lambda;\calg_\Lambda\cup \calp_\Lambda,\calh_\Lambda^t).\]
The restriction maps to each factor are surjective, however, $R$ is not surjective.
For example, consider the partial conjugations $\pi^v_x$ and $\pi^v_w$.
We can see that $([\pi^v_w],[\pi^v_x])$ is in the image of $R$, since $R([\pi^v_K])=([\pi^v_w],[\pi^v_x])$, where $K$ is the subgraph on $\{w,x\}$.
In particular, $[\pi^v_x]$ is in $\out^0(A_\Lambda;\calg_\Lambda\cup \calp_\Lambda,\calh_\Lambda^t)$.
However, $(1,[\pi^v_x])$ is not in the image of $R$: if $\phi$ were an automorphism mapping to this pair, the requirements that $\phi(w)$ commute with $\phi(x)$, $\phi(y)$ commute with $\phi(z)$, $v\notin\crsupp(\phi(wy))$, and $v\in\crsupp(\phi(xz))$ are inconsistent.

To see what goes wrong, consider applying the restriction maps $R_\Delta$ and $R_\Lambda$ serially.
We have
\[R_\Delta\colon \outgh \to \out^0(A_\Delta;\calg_\Delta\cup \calp_\Delta,\calh_\Delta^t);\]
it is surjective and its kernel is $\out^0(A_\G; \calg, \{A_\Xi,A_\Delta\}^t)$.
Now, to correctly compute the image of $R_\Lambda$ on this kernel, we must add $A_{\Delta-\{v\}}$, $A_{\Delta-\{x\}}$, and $A_{\Delta-\{z\}}$ to $\calg$.
Let $\calg'$ denote this expanded version of $\calg$.
Then we have a surjective map
\[R_\Lambda\colon \out^0(A_\G; \calg, \{A_\Xi,A_\Delta\}^t) \to \out^0(A_\Lambda;\calg'_\Lambda\cup \calp'_\Lambda,\calh_\Lambda^t),\]
where $\calp'_\Lambda$ is computed using $\calg'$.
Then $\genby{x,z}$ is in $\calp'_\Lambda$, since this is $\genby{N_{(\calg')^v}(\{w,y\})\cap \Lambda}$, and $\{w,y\}$ is $
(\calg')^v$-connected since $A_{\Delta-\{v\}}\in\calg'$.
Then $[\pi^v_x]$ is not in the image of $R_\Lambda$, since it does not act trivially on $\genby{x,z}$.
Notice that $[\pi^v_x]$ being in the image of $R_\Lambda$ on the kernel of $R_\Delta$ is the same as $(1,[\pi^v_x])$ being in the image of $R$.
This explains why $(1,[\pi^v_x])$ is not in the image of $R$.
\end{example}

\subsubsection{Absolute automorphism groups leading to relative ones}
For any pair $(\G,\calg)$, we can find an extension $\Lambda$ of $\G$ where $\G$ is invariant in $\Lambda$ and
\[R_\G\colon \out^0(A_\Lambda) \to \out^0(A_\G;\calg)\]
is surjective.
In fact, we can take this $\Lambda$ to be the relative cone graph $\widehat \G$ of $\G$ with respect to $\calg$, as in the proof of Proposition~\ref{pr:generating1}.
By taking kernels under further restriction maps, we can get the group $\outgh$ for any $\calh\subset\calg$ we please.
This means that for every triple $(\G,\calg,\calh)$, there is an absolute automorphism group $\out^0(A_\Lambda)$, where $\outgh$ arises in the study of $\out^0(A_\Lambda)$ using restriction maps.
In particular, this means that all of the base cases in the proof of Theorem~\ref{th:vf} actually occur.

\subsection{Computations of virtual cohomological dimension}\label{ss:vcdeg}
Recall that the \emph{cohomological dimension} $\cd(G)$ of $G$ is the supremum of degrees in which $G$ has nontrivial group cohomology, over all coefficient modules.
The \emph{virtual cohomological dimension} of $G$, denoted $\vcd(G)$, is the cohomological dimension of any torsion-free finite index subgroup of $G$.

We review a few facts about these dimensions.
\begin{itemize}
\item 
If $G$ has a classifying space of dimension $n$, then $\cd(G)\leq n$.
\item
Suppose we have an exact sequence
\[1\to H \to G\to Q \to 1.\]
Then $\cd(G)\leq \cd(H)+\cd(Q)$.
\item If $H\leq G$, then $\cd(H)\leq\cd(G)$.
\item If $N$ is finitely-generated nilpotent, then $\vcd(N)$ is equal to the sum of the $\mathbb{Q}$-ranks of the subquotient factors in the lower central series of $N$.
\end{itemize}
These can be found in standard references, such as Brown~\cite{Brown}.
We also need the following theorem, which is closely related to Theorem~\ref{th:F_from_rigid_actions}.
\begin{theorem}[Geoghegan~\cite{MR2365352}, Theorem~7.3.3]\label{th:vcdbyaction}
Suppose a group $G$ acts cocompactly by rigid homeomorphisms on a contractible CW complex $X$ such the stabilizer of each $i$-cell has cohomological dimension less than or equal to $d_i$.
Then the cohomological dimension of $G$ is less than or equal to $\max_i(i+d_i)$.
\end{theorem}

The exact sequence from Theorem~\ref{th:restriction} can be used to prove upper bounds on the $\vcd(\out(A_\G))$.
Technically, we are bounding cohomological dimension of $\out^{[l]}(A_\G)$, and using the exact sequence from Theorem~\ref{th:pcongrestriction}.

On the other hand, these exact sequences give us hints about where to look for generators for nilpotent subgroups of $\out(A_\G)$ that we can use to find lower bounds for $\vcd(\out(A_\G))$.
We give computations of vcd for two families of groups.
For the first, we verify a previously-known upper bound by different methods.
For the second, we find a new upper bound.
In both, we find lower bounds by building nilpotent subgroups.

We hope that these techniques can be generalized to give algorithms for computing vcd of $\out(A_\G)$ for general $\G$.

\subsubsection{The string of diamonds}
Let $A_d$ be the right-angled Artin group given by a \emph{string of $d$ diamonds}, studied in Section 5.3 of Charney--Stambaugh--Vogtmann~\cite{MR3626599} and depicted in Figure~\ref{fi:diamonds}. 
Then $\out(A_d)$ is an untwisted automorphism group. 
It is shown in~\cite{MR3626599} that the dimension of the spine of the associated Outer space has dimension $4d-1$, which gives an upper bound for $\vcd(\out(A_d))$. 

\begin{figure}
\begin{centering}
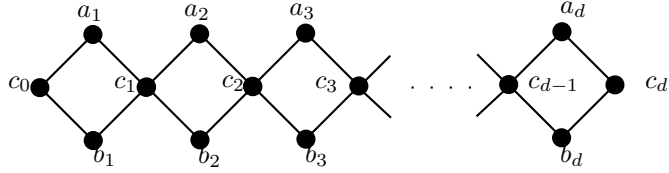
\end{centering}\caption{The string of $d$ diamonds, with vertex labels.}\label{fi:diamonds}
\end{figure} 

We give an alternative proof of the upper bound using our exact sequence, and find a free abelian group of rank equal to $4d-1$ in $\out(A_d)$, which shows that this bound is sharp. After completing this work, we learned that Millard also proved that $\vcd(\out(A_d))=4d-1$ by exhibiting a free abelian subgroup of the appropriate rank \cite{Millard}. 
For convenience, we use $\pi^x_{y}$ to denote the partial conjugation by $x$ on the component of $\G - \st(x)$ containing $y$.
\begin{proposition}
When $d=1$ we have 
\[ \vcd(\out(A_1))=\vcd(\out(A_1;\{\genby{c_1}\}))=\vcd(\out(A_1;\{\genby{c_1}\}^t))=2,\text{ and }\]
\[\vcd(\out(A_1;\{\genby{c_0}\},\{\genby{c_1}\}^t))=1.\]
and for $d \geq 2$ we have \begin{align*} \vcd(\out(A_d)) &=4d-1 \\ \vcd(\out(A_d;\{\genby{c_d}\}^t))&= 4d-2. \end{align*} 
\end{proposition}

\begin{proof}
First we prove the upper bound inductively.
If $d=1$, $\out(A_1)$ is $\out(\mathbb{F}_2 \oplus \mathbb{F}_2)$, which has a finite index subgroup isomorphic to $\out(\mathbb{F}_2) \oplus \out(\mathbb{F}_2)$.
The groups $\out(A_1;\{\genby{c_1}\})$ and $\out(A_1;\{\genby{c_1}\}^t)$ have a finite index subgroup isomorphic to $\out(\mathbb{F}_2) \oplus \Z$, where the cyclic factor is generated by $\rho^{c_1}_{c_0}$.
The group $\out(A_1;\{\genby{c_0}\},\{\genby{c_1}\}^t)$ is virtually $\out(\mathbb{F}_2)$.
 As $\out(\mathbb{F}_2)$ is virtually free, this gives the correct upper bound for the vcd in all subcases.
Abelian subgroups giving the lower bounds are easy to find in the case $d=1$.

Now we consider general $d$.
We consider $A_{d-1}$ as a special subgroup of $A_d$ in the obvious way.
Let $B$ be the special subgroup $\genby{c_{d-1},a_d,b_d,c_d}$ of $A_d$; then $B\cong A_1$.
By Proposition~\ref{pr:invspecial}, both $A_{d-1}$ and $B$ are invariant special subgroups in $\out^0(A_d)$.
By Theorem~\ref{th:restriction} and Proposition~\ref{pr:easier_way_to_find_peripheral_structure}, we have exact sequences
\[1\to \out^0(A_d;\{B\}^t)\to\out^0(A_d)\to \out^0(B;\{\genby{c_{d-1}}\})\to 1, \text{ and }\]
\[1\to \out^0(A_d;\{B,A_{d-1}\}^t)\to \out^0(A_d;\{B\}^t)\to \out^0(A_{d-1};\{\genby{c_{d-1}}\}^t)\to 1.\]
The inner kernel is generated by the elements $\pi^{a_d}_{b_d}, \pi^{b_d}_{a_d}, \pi^{a_{d-1}}_{c_{d}}, \pi^{b_{d-1}}_{c_{d}}$, and $\pi^{c_{d-1}}_{c_0}$, using Theorem~\ref{th:generators}.
This is isomorphic to $\mathbb{F}_2 \oplus \mathbb{F}_2 \oplus \mathbb{Z}$, so is 3-dimensional.
This isomorphism is a little tricky to see, but it can be shown using the restriction map to $A_{\st(c_{d-1})}$ and the projection map to $A_{\lk(c_{d-1})}$. 
Inductively, $\vcd(\out^0(A_{d-1};\{\genby{c_{d-1}}\}^t))=4d-6$ and $\vcd(\out^0(B;\{\genby{c_{d-1}}\}))=2$.
The two exact sequences give us an upper bound of $4d-1=4d-6+3+2$ for $\vcd(\out(A_d))$.

The argument for $\out(A_d;\{\genby{c_d}\}^t)$ is quite similar.
We get the same exact sequences, except that now we demand that all groups act trivially on $\genby{c_d}$.
We have $\vcd(\out^0(B;\{\genby{c_{d-1}}\},\{\genby{c_d}\}^t))=1$, and our upper bound is $4d-2=4d-6+3+1$.

Having shown the upper bound, we use the groups in our exact sequences as places to look for generators for an abelian subgroup that will give use a lower bound for the dimension.
In the inner kernel above, we pick three generators giving us a copy of $\Z^3$: 
$\pi^{a_d}_{b_d}$, $\pi^{a_{d-1}}_{c_{d}}$, and $\pi^{c_{d-1}}_{c_0}$.
From $\out^0(B;\{\genby{c_{d-1}}\})$, we lift two generators to $\out^0(A_d)$: $\rho^{a_d}_{b_d}$ and $\rho^{c_{d-1}}_{c_d}$.
We recursively select similar generators from $\out^0(A_{d-1};\{\genby{c_{d-1}}\}^t)$ and lift them to $\out^0(A_d)$.
To summarize, we have a subgroup $G$ generated by:
\begin{itemize}
 \item $\rho^{a_i}_{b_i}$ for $1 \leq i \leq d$.
 \item $\pi^{a_i}_{b_i}$ for $1 \leq i \leq d$.
 \item $\pi^{a_i}_{c_d}$ for $1 \leq i \leq d$, $i \neq 1$ and $i \neq d$.
 \item $\pi^{c_i}_{c_0}$ for $1 \leq i \leq d-1$.
 \item $\rho^{c_1}_{c_0}$ and $\rho^{c_{d-1}}_{c_d}$.
\end{itemize}
One can show that all of these elements commute in $\out(A_d)$. 
This is because either the supports of the automorphisms are disjoint, or a partial conjugation is involved and one can use the trick that $\pi^x_K = (\pi^x_{\G -K})^{-1}$ in $\out(A_d)$ to reduce to the case where the supports are disjoint. 
Since $G$ is an abelian group, it makes sense to ask whether the generators are linearly independent.
In fact, these generators are linearly independent, as we now explain. 
One sees that the image of the abelian group $G$ in $\gln$ under the action on the abelianization is free abelian of rank $d +2$. 
Furthermore, in the kernel ${\rm{IA}}_\G$, the fact that the remaining $3d-3$ partial conjugations are linearly independent can be seen using the first Johnson homomorphism on ${\rm{IA}}_\G$ (see Section~4 of \cite{W1}).
So $G$ is isomorphic to $\Z^{4d-1}$, proving the lower bound.

For the group $\out(A_d;\{\genby{c_d}\}^t)$, we use the subgroup generated by the same generators as $G$, except that we leave out $\rho^{c_{d-1}}_{c_d}$.
This gives the correct lower bound.
\end{proof}

\subsubsection{Graphs whose vertex class graph is a $4$-path}
Let $p,q,r,s$ be positive integers and let $\G$ be the graph whose vertex class graph $\overline{\G}$ is the four-vertex path with coloring $(p,1)$, $(q,0)$, $(r,0)$, and $(s,1)$.
(If $p=1$ or $s=1$, replace the corresponding label with $(1,0)$.)
Concretely, $\G$ is a graph with four equivalence classes; call them $[w]$, $[x]$, $[y]$, and $[z]$.
We label these so that $\abs{[w]}=p$, $\abs{[x]}=q$, $\abs{[y]}=r$, and $\abs{[z]}=s$.
The classes $[w]$ and $[z]$ generate free groups, and the classes $[x]$ and $[y]$ are complete subgraphs.
All possible edges appear between $[w]$ and $[x]$, between $[x]$ and $[y]$, and between $[y]$ and $[z]$, and there are no other edges.

If $q$ or $r$ is greater than $1$, then $\G$ contains triangles.
This $A_\G$ does not decompose as a free or direct product.
This $\out(A_\G)$ contains both adjacent and nonadjacent transvections.
Taken together, this excludes most previous techniques for bounding the dimension of $\out(A_\G)$.

\begin{figure}
\begin{centering}
\includegraphics{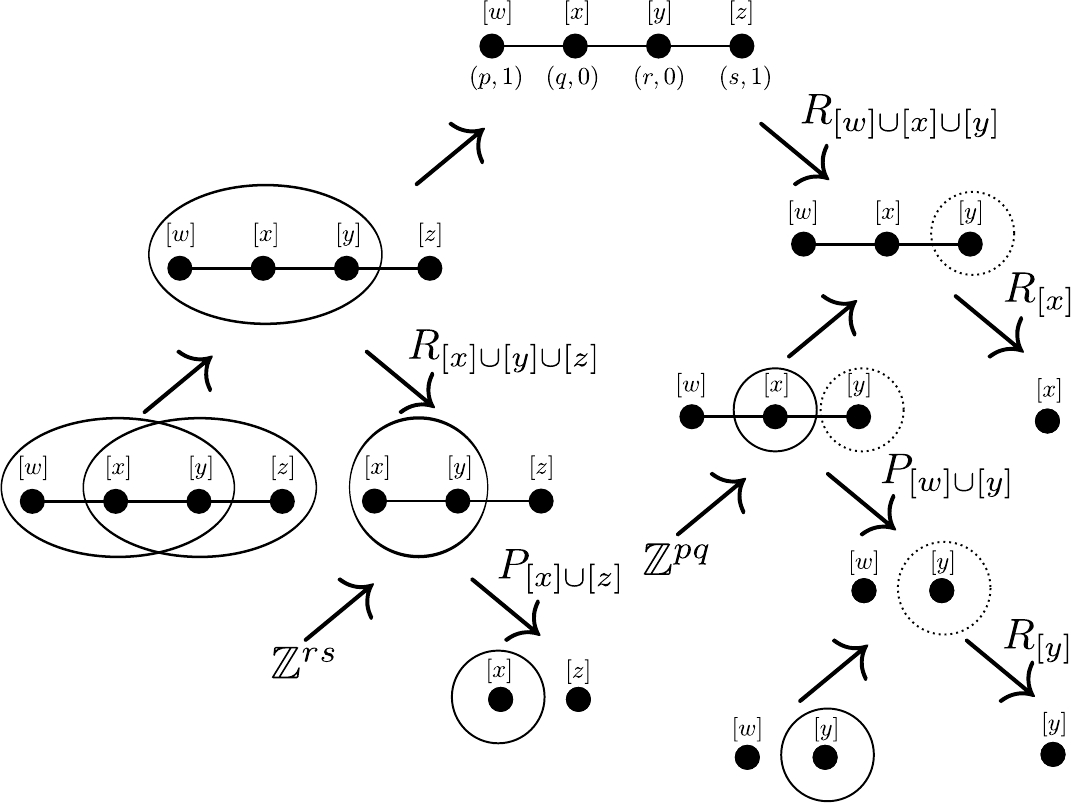}
\end{centering}
\caption{The four restriction maps and two projection maps in dismantling $\out^0(A_\G)$, when $\overline{\G}$ is the initial graph pictured. Dashed ovals represent groups in the peripheral structure that must be left invariant, and solid ovals represent those that are acted on trivially.}
\label{fig:vcgpath}
\end{figure}

\begin{proposition}
Let $\G$ be the graph on $p+q+r+s$ vertices described above.
Then 
\[\vcd(\out(A_\G))=\frac{q(q-1)}{2}+\frac{r(r-1)}{2}+rs+pq+q(2s-1)+r(2p-1).\]
\end{proposition}

\begin{proof}
The upper bound comes from six applications of exact sequences.
We summarize this in Figure~\ref{fig:vcgpath}.
Each restriction map restricts to the star of a maximal vertex class or a maximal vertex class itself.
These are invariant by Proposition~\ref{pr:invspecial}, so these maps are well defined.
Once we have forced central vertices to act trivially, we project to automorphisms of the graph where these central vertices have been deleted.
We inspect leaves of the tree in the figure; these are the final kernels and images that we can take apart no further, and adding up bounds for the vcd of each leaf will give a bound for the group. The unlabeled arrows on the diagram are the kernels of the adjacent restriction or projection maps on the same level.

First we consider the projection maps.
The projection kernels are generated by transvections with central acting letters.
These are the $\Z^{pq}$ and $\Z^{rs}$ in the figure, and contribute $pq$ and $rs$ to the dimension.
Next we consider the groups $\out^0(\genby{[x]})$ and $\out^0(\genby{[y]})$, which appear as images of restriction maps.
These are simply $\gl(q,\Z)$ and $\gl(r,\Z)$.
By Borel--Serre~\cite{MR0387495} (or see Brown~\cite{Brown}, Example VIII.9.3), their vcds are $q(q-1)/2$ and $r(r-1)/2$.
Next we consider the innermost kernel.
This is $\out^0(A_\G;\{\genby{[w],[x],[y]},\genby{[x],[y],[z]}\}^t)$.
By Theorem~\ref{th:generators}, this is trivial.

Finally, we have the free products.
We focus on $\out^0(\genby{[x],[z]};\{\genby{[x]}\}^t)$; this is a Fouxe-Rabinovitch group.
Let $G=\genby{[x],[z]}$ and let $\calg=\{\genby{[x]}\}$.
Let $\spine$ be the spine of the Guirardel--Levitt outer space for $G$ with respect to $\calg$.
Each simplex in this spine represents a minimal, simplicial action of $G$ on a tree, with vertex stabilizers trivial or conjugate to $\genby{[x]}$.
Let $T$ be a tree representing a simplex in $\spine$.
The stabilizer $\out(G;\calg^t)_T$ of $T$ has a finite index subgroup that acts trivially on $T/G$.
As explained in the proof of Proposition~3.7 in Guirardel--Levitt~\cite{GL}, this finite index subgroup is isomorphic to $\genby{[x]}^{k-1}$, where $k$ is the degree of the vertex in $T/G$ with local group $\genby{[x]}$.
If $T/G$ is a rose (a graph with exactly one vertex), then $k=2s$ and $T$ is a $0$-simplex of $\spine$.
This gives $q(2s-1)$ as $i+d_i$, where $T$ is an $i$-simplex and $d_i=\vcd(\out(G;\calg^t)_T)$. 
If $T$ is an $i$-simplex of $\spine$, then adding $i$ edges to a rose means the valence of the vertex with stabilizer $\genby{[x]}$ is at most $k\leq 2s-i$, and $\vcd(\out(G;\calg^t)_T)\leq q(2s-i-1)$.
This gives $q(2s-i-1)+i$ as an upper bound for $i+d_i$.
In any case, $q(2s-1)$ is an upper bound for $i+d_i$ over all simplices of $\spine$, and by Theorem~\ref{th:vcdbyaction}, we have $\vcd(\out^0(\genby{[x],[z]};\{\genby{[x]}\}^t))\leq q(2s-1)$.
Similarly, $r(2p-1)$ is an upper bound for $\vcd(\out^0(\genby{[y],[w]};\{\genby{[y]}\}^t))$.
Summing these bounds over all seven images and kernels gives our upper bound.

For the lower bound, we totally order the vertices of $\G$ as $v_1,\dotsc,v_n$, and form a group $G$ generated by the following:
\begin{itemize}
\item $\rho^{v_i}_{v_j}$ where $v_i, v_j\in [x]$ and $i<j$ ($q(q-1)/2$ generators),
\item $\rho^{v_i}_{v_j}$ where $v_i, v_j\in [y]$ and $i<j$ ($r(r-1)/2$ generators),
\item $\rho^{v}_{u}$ where $v \in [x]$ and $u\in [w]$ ($pq$ generators),
\item $\rho^{v}_{u}$ where $v \in [y]$ and $u\in [z]$ ($rs$ generators),
\item $\rho^{v}_{u}$ where $v \in [x]$ and $u\in [z]$ ($qs$ generators),
\item $\pi^{v}_{u}$ where $v \in [x]$ and $u\in [z]$, leaving out the last $u$ for each $v$ ($q(s-1)$ generators),
\item $\rho^{v}_{u}$ where $v \in [y]$ and $u\in [w]$ ($rp$ generators), and
\item $\pi^{v}_{u}$ where $v \in [y]$ and $u\in [w]$, leaving out the last $u$ for each $v$ ($r(p-1)$ generators).
\end{itemize}
It is straightforward to check that for any two generators in this list, either they commute or their commutator is another generator in this list.
Therefore $G$ is nilpotent.
Further, at each level of the lower central series of $G$, the generators contained there are linearly independent in their corresponding subquotient.
This can be verified by inspecting the action of $G$ on the abelianization of $A_\G$, and the image of $\ia\cap G$ under the Johnson homomorphism.
Therefore the total count of generators here is $\vcd(G)$.
This gives the required lower bound on $\vcd(\out(A_\G))$.
\end{proof}

\bibliography{restrictionbib}
\bibliographystyle{abbrv}

\vspace{0.2cm}
\small
\sc \noindent Matthew B. Day, Department of Mathematical Sciences, 309 SCEN, University of Arkansas, Fayetteville, AR 72701, U.S.A. \\
\noindent \tt e-mail: matthewd@uark.edu
\\
\\
\sc \noindent Richard D. Wade, Mathematical Institute, Andrew Wiles Building, University of Oxford,  Woodstock Road, Oxford,
OX2 6GG, U.K. \\
\noindent \tt e-mail: wade@maths.ox.ac.uk

\end{document}